\documentclass{article}
\usepackage[utf8]{inputenc}
\usepackage[T1]{fontenc}
\usepackage[a4paper,margin=2.25cm]{geometry}
\usepackage{lmodern}

\usepackage{graphicx}
\usepackage{amsmath,amssymb,amsfonts}
\usepackage{amsthm}
\usepackage{algorithm}
\usepackage{algpseudocode}
\usepackage{multirow}
\usepackage{rotating}
\usepackage{caption}
\usepackage{subcaption}
\usepackage{xcolor}
\usepackage[hidelinks]{hyperref}
\usepackage{booktabs}
\usepackage{comment}
\usepackage{tabstackengine}
\usepackage{appendix}

\newtheorem{proposition}{Proposition}

\newif\ifisThesis

\DeclareMathOperator{\chol}{chol}

\newcommand{\algPotrf}{\texttt{potrf} \hspace{1pt} $\frac{n^3}{3}$}
\newcommand{\algTrsm}{\texttt{trsm} \hspace{2.5mm} $n^3$}
\newcommand{\algSyrk}{ \texttt{syrk} \hspace{2.5mm} $n^3$}
\newcommand{\algGemm}{\texttt{gemm} \hspace{0.7mm} $2n^3$}

\title{GPU-Accelerated Cholesky Factorization of Block Tridiagonal Matrices}
\author{Roland Schwan \and Daniel Kuhn \and Colin N. Jones%
\thanks{Large language models were used to improve the writing quality in the manuscript, including grammar, clarity, and flow.
}%
\thanks{This work was supported as a part of NCCR Automation, a National Centre of Competence in Research, funded by the Swiss National Science Foundation (grant number 51NF40\_225155).
}%
\thanks{Roland Schwan and Colin N. Jones are with the Automatic Control Lab, EPFL, Switzerland. Roland Schwan and Daniel Kuhn are with the Risk Analytics and Optimization Chair, EPFL, Switzerland. \tt{\{roland.schwan, daniel.kuhn, colin.jones\}@epfl.ch}}%
}
\date{}

\begin{document}

\makeatletter
\let\@fnsymbol@orig\@fnsymbol
\def\@fnsymbol#1{}
\makeatother

\maketitle

\makeatletter
\let\@fnsymbol\@fnsymbol@orig
\makeatother

\begin{abstract}
This paper presents a GPU-accelerated framework for solving block tridiagonal linear systems that arise naturally in numerous real-time applications across engineering and scientific computing. Through a multi-stage permutation strategy based on nested dissection, we reduce the computational complexity from $\mathcal{O}(Nn^3)$ for sequential Cholesky factorization to $\mathcal{O}(\log_2(N)n^3)$ when sufficient parallel resources are available, where $n$ is the block size and $N$ is the number of blocks. The algorithm is implemented using NVIDIA's Warp library and CUDA to exploit parallelism at multiple levels within the factorization algorithm. Our implementation achieves speedups exceeding 100x compared to the sparse solver QDLDL, 25x compared to a highly optimized CPU implementation using BLASFEO, and more than 2x compared to NVIDIA's CUDSS library. The logarithmic scaling with horizon length makes this approach particularly attractive for long-horizon problems in real-time applications. Comprehensive numerical experiments on NVIDIA GPUs demonstrate the practical effectiveness across different problem sizes and precisions. The framework provides a foundation for GPU-accelerated optimization solvers in robotics, autonomous systems, and other domains requiring repeated solution of structured linear systems. The implementation is open-source and available at \url{https://github.com/PREDICT-EPFL/socu}.
\end{abstract}

\section{Introduction}

Block tridiagonal linear systems arise naturally in numerous real-time applications across engineering and scientific computing. In model predictive control (MPC), trajectory optimization, and Kalman filtering, the temporal structure of the underlying dynamics leads to block tridiagonal systems that must be solved repeatedly at high frequencies. Examples span diverse domains, including legged robotics \cite{amatucci2025}, Formula One race car control \cite{perantoni2014}, train trajectory optimization \cite{kouzoupis2023}, and spacecraft trajectory planning \cite{wang2018}. These applications increasingly demand long prediction horizons to improve performance and handle complex constraints, yet despite the inherent sparsity of these systems, sequential solvers remain a computational bottleneck, limiting the achievable sampling rates and horizon lengths in real-time applications.

This \ifisThesis chapter \else paper \fi presents a GPU-accelerated Cholesky factorization framework specifically designed for block tridiagonal matrices. Unlike existing approaches that modify the optimization algorithm itself, we introduce parallelization directly at the linear algebra level. For problems with block size $n$ and horizon $N$, through a nested dissection permutation strategy, we reduce the computational complexity from $\mathcal{O}(Nn^3)$ for sequential factorization to $\mathcal{O}(\log_2(N)n^3)$ when sufficient parallel resources are available. This logarithmic scaling is particularly advantageous for long-horizon problems, where the performance gap compared to sequential methods grows substantially with increasing horizons. The approach preserves the numerical properties and robustness of the underlying solver, making it broadly applicable across various optimization frameworks.

Existing approaches to parallelizing block tridiagonal solvers face significant limitations. Sequential structure-exploiting solvers achieve optimal floating-point operation counts but are fundamentally constrained by $\mathcal{O}(Nn^3)$ complexity due to inherent data dependencies \cite{schwan2025}. General parallel sparse factorization methods \cite{heath1991} do not fully exploit the regular block structure, while partition-based parallelization strategies \cite{cao2002} achieve only modest speedups due to substantial sequential overhead.

Parallelization techniques for model predictive control have been explored previously, achieving similar $\mathcal{O}(\log_2(N)n^3)$ complexity \cite{nielsen2014, nielsen2015, sarkka2023}, with applications to robotics \cite{jallet2024, amatucci2025}. However, these methods are extremely tailored to MPC and do not accommodate more general problem structures. They decompose the optimization problem itself, tying the approach to the specific structure of MPC. In contrast, our approach follows the ideas from our previous work \cite{schwan2025} and parallelizes directly at the linear algebra level. This results in reduced overhead and enables future extensions to more general structures.

The methods presented in this paper are directly applicable to the growing interest in deploying optimization algorithms for robotics applications. Recent years have seen increasing interest in GPU-based solvers that solve the underlying linear systems using preconditioned conjugate gradient (PCG) methods~\cite{adabag2024, yang2025, adabag2025}. While these iterative methods offer excellent parallelizability, their convergence depends critically on the system's conditioning. Consequently, accuracy is typically limited, and the required number of iterations grows with problem dimensionality.

In recent years, combining model predictive control with reinforcement learning (RL) has gained popularity in robotics applications \cite{jenelten2024, jeon2025}; a detailed survey can be found in \cite{reiter2025}. While this combination can drastically reduce the number of samples required for learning, it comes at the cost of increased per-sample computational expense. This is primarily due to inefficient solving of the underlying optimization problem, which often remains slow. In contrast, purely neural network-based policies are fully batched and executed on GPUs. In this work, we aim to close this gap by developing tools that leverage GPU acceleration to make MPC competitive for RL applications.

The main contributions of this \ifisThesis chapter \else paper \fi are: (1) a multi-stage permutation strategy based on nested dissection that achieves logarithmic complexity on GPUs, (2) operation-level parallelism through CUDA streams and atomic operations that further reduces the critical path length, (3) fused and blocked kernel implementations optimized for different problem sizes and memory constraints, and (4) comprehensive numerical experiments demonstrating speedups of 100x--500x compared to the sparse solver QDLDL, 25x--40x compared to a highly optimized CPU implementation using the BLASFEO library, and more then 2x compared to the closed-source CUDSS library by NVIDIA.

The remainder of this \ifisThesis chapter \else paper \fi is organized as follows. Section~\ref{sec:block_cholesky} reviews three variants of block Cholesky factorization and their memory access patterns. Section~\ref{sec:sequential} presents the sequential baseline algorithm and analyzes its computational complexity. Section~\ref{sec:parallel} develops our parallelization strategies, progressing from partition-based approaches to single-stage and finally multi-stage permutations. Section~\ref{sec:numerical_implementation_gpu} details the GPU implementation, including the CUDA programming model, kernel fusion strategies, and memory optimization techniques. Section~\ref{sec:numerical_examples_gpu} presents comprehensive numerical experiments comparing our approach against state-of-the-art CPU and GPU-based solvers across various problem sizes and precisions. Finally, Section~\ref{sec:conclusion} concludes with a summary and discussion of future research directions.

\section{Block Cholesky Factorization} \label{sec:block_cholesky}

Consider a positive-definite block matrix $A \in \mathbb{R}^{Nn \times Nn}$ with $N \times N$ blocks of size $n\times n$, where each block is indexed as $A_{ij}$. The dense Cholesky factorization fundamentally consists of a sequence of Gaussian elimination steps based on the elementary operation
\begin{equation*}  
    A_{ij} \leftarrow A_{ij} - \left( A_{ik} A_{jk}^\top \right) A_{kk}^{-\top}  
\end{equation*}  
The loop indices $i$, $j$, and $k$ can be nested in any order, yielding different memory access patterns and giving rise to three distinct Cholesky factorization algorithms illustrated in Figure~\ref{fig:cholesky_access_viz}:
\begin{enumerate}  
  \item \textbf{Submatrix-Cholesky} (Right-Looking): Columns are computed sequentially in the outer loop, with the right-looking submatrix updated in the inner loop. This approach is also known as column \textit{immediate-update} or \textit{fan-out}.  
  \item \textbf{Column-Cholesky} (Left-Looking): Columns are computed sequentially, but each column is updated in a delayed fashion using previously computed columns to the left. This method is also referred to as column \textit{delayed-update} or \textit{fan-in}.  
  \item \textbf{Row-Cholesky} (Up-Looking): Rows are computed sequentially, with updates based on previously computed rows above. This algorithm is also known as row \textit{delayed-update} or \textit{fan-in}.  
\end{enumerate}  
Each algorithm offers distinct advantages and disadvantages regarding memory access patterns and corresponding optimized data structures, particularly in sparse matrix scenarios. Further details can be found in \cite{heath1991}. The algorithms developed in subsequent sections leverage a combination of submatrix- and column-Cholesky approaches to achieve parallelization while avoiding race conditions.

\begin{figure}[ht]
\centering
\includegraphics[width=1.0\textwidth]{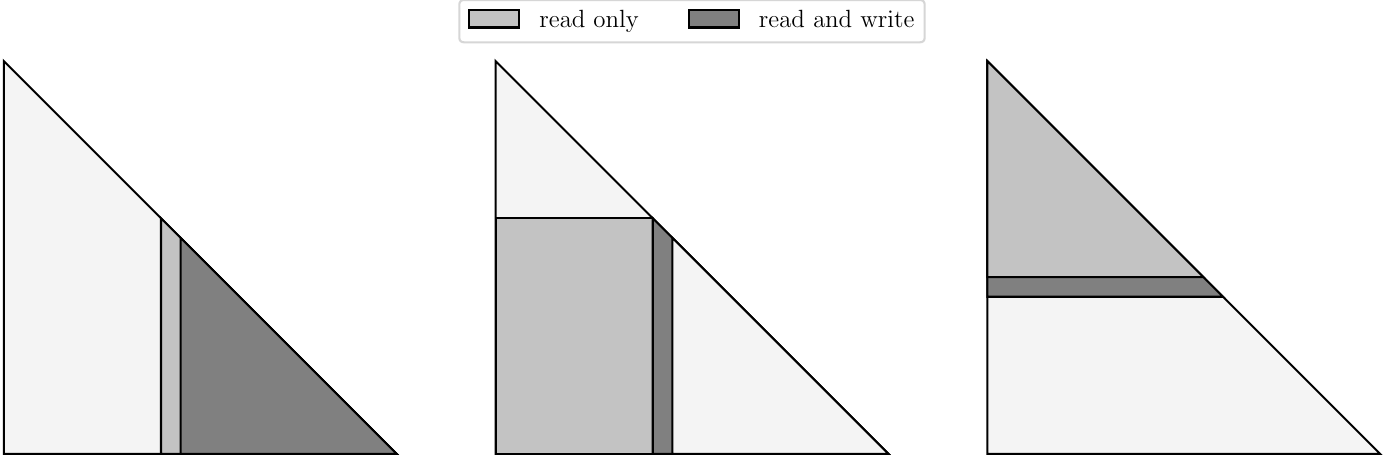}
\begin{minipage}[t]{0.32\textwidth}
\vspace{-5mm}
\begin{algorithm}[H]
\caption*{Submatrix-Cholesky}
\label{alg:submatrix_cholesky}
\begin{algorithmic}[1]
\For{$k = 1$ to $n$}
    \State $A_{kk} \leftarrow \text{chol}(A_{kk})$
    \For{$i = k + 1$ to $n$}
        \State $A_{ik} \leftarrow A_{ik} A_{kk}^{-T}$
    \EndFor
    \For{$j = k + 1$ to $n$}
        \For{$i = j$ to $n$}
            \State \hspace{-3mm}$A_{ij} \mathrel{-}\mkern-1mu= A_{ik} A_{jk}^T$
        \EndFor
    \EndFor
\EndFor
\end{algorithmic}
\end{algorithm}
\end{minipage}
\hfill
\begin{minipage}[t]{0.32\textwidth}
\vspace{-5mm}
\begin{algorithm}[H]
\caption*{Column-Cholesky}
\label{alg:column_cholesky}
\begin{algorithmic}[1]
\For{$k = 1$ to $n$}
    \For{$j = 1$ to $k - 1$}
        \For{$i = k$ to $n$}
            \State \hspace{-3mm}$A_{ik} \mathrel{-}\mkern-1mu= A_{ij} A_{kj}^T$
        \EndFor
    \EndFor
    \State $A_{kk} \leftarrow \text{chol}(A_{kk})$
    \For{$i = k + 1$ to $n$}
        \State $A_{ik} \leftarrow A_{ik} A_{kk}^{-T}$
    \EndFor
\EndFor
\end{algorithmic}
\end{algorithm}
\end{minipage}
\hfill
\begin{minipage}[t]{0.32\textwidth}
\vspace{-5mm}
\begin{algorithm}[H]
\caption*{Row-Cholesky}
\label{alg:row_cholesky}
\begin{algorithmic}[1]
\For{$i = 1$ to $n$}
    \For{$j = 1$ to $i$}
        \For{$\!k\!=\!\!1\!$ to $\!j\!-\!1$}
            \State \hspace{-3mm}$A_{ij} \mathrel{-}\mkern-1mu= A_{ik} A_{jk}^T$
        \EndFor
        \If{$i = j$}
        \State $\!\!\!\!\!A_{jj} \!\leftarrow\! \text{chol}(A_{jj})$
        \Else
            \State $\!\!\!\!\!A_{ij} \!\leftarrow\! A_{ij} A_{jj}^{-T}$
        \EndIf
    \EndFor
\EndFor
\end{algorithmic}
\end{algorithm}
\end{minipage}
\caption{Three Cholesky factorization variants and their memory access patterns.}
\label{fig:cholesky_access_viz}
\end{figure}

\section{Sequential Cholesky Factorization} \label{sec:sequential}

We consider symmetric positive definite block tridiagonal matrices of the form
\begin{equation*}
\Psi = \begin{bmatrix}
D_1 & E_1^\top & & & \\[1ex]
E_1 & D_2 & E_2^\top & & \\
& E_2 & \ddots & \ddots & \\
& & \ddots & D_{N-1} & E_{N-1}^\top \\[1ex]
& & & E_{N-1} & D_N
\end{bmatrix}
\end{equation*}
where $D_i \in \mathbb{R}^{n \times n}$ are symmetric diagonal blocks and $E_i \in \mathbb{R}^{n \times n}$ are off-diagonal blocks. The Cholesky factorization $\Psi = LL^\top$ preserves the block tridiagonal structure, yielding  
\begin{equation*}
L = \begin{bmatrix}
\hat{D}_1 & & & & \\[1ex]
\hat{E}_1 & \hat{D}_2 & & & \\
& \hat{E}_2 & \ddots & & \\
& & \ddots & \hat{D}_{N-1} & \\[1ex]
& & & \hat{E}_{N-1} & \hat{D}_N
\end{bmatrix}
\end{equation*}
where $\hat{D}_i$ are lower triangular diagonal blocks and $\hat{E}_i$ are the corresponding factorized off-diagonal blocks.

\begin{algorithm}[t]
\caption{Sequential factorization of $\Psi$}
\label{alg:seq_factorization}
\begin{algorithmic}[1]
\State $\hat{D}_1 \leftarrow \chol(D_1)$ \hfill \algPotrf
\For{$i = 2,\ldots,N$}
    \State $\hat{E}_{i-1} \leftarrow E_{i-1} \hat{D}_{i-1}^{-\top}$ \hfill \algTrsm
    \State $\hat{D}_i \leftarrow D_i - \hat{E}_{i-1} \hat{E}_{i-1}^\top$ \hfill \algSyrk
    \State $\hat{D}_i \leftarrow \chol(\hat{D}_i)$ \hfill \algPotrf
\EndFor
\end{algorithmic}
\end{algorithm}

Algorithm~\ref{alg:seq_factorization} presents the standard left-looking sequential factorization. The algorithm begins by factorizing the first diagonal block, then processes each subsequent block row by solving for the off-diagonal block via triangular solve (\texttt{trsm}), updating the diagonal block with a symmetric rank-$k$ update (\texttt{syrk}), and computing its Cholesky factorization (\texttt{potrf}).

This algorithm achieves optimal floating-point operation count by preserving the sparse structure without introducing fill-in elements. However, it suffers from inherent sequential dependencies between consecutive blocks. Each block $i$ cannot begin processing until block $i-1$ is fully factorized, as the \texttt{trsm} operation requires the completed factorized diagonal block $\hat{D}_{i-1}$ from the preceding iteration.

\subsection{Computational Complexity}

\begin{table}
\centering
\caption{Cost of elementary matrix operations.}
\begin{tabular}{lll}
\multicolumn{3}{c}{$A : m \times n, B : n \times p, D, L : n \times n, L$ lower triangular.} \\
\toprule
Operation & Kernel & Cost (flops) \\
\midrule
Matrix matrix multiplication $A \cdot B$ & \texttt{gemm} & $2mnp$ \\
Symmetric matrix multiplication $A \cdot A^T$ & \texttt{syrk} & $m^2n$ \\
Cholesky decomposition s.t. $LL^T = D$ & \texttt{potrf} & $1/3n^3$ \\
Solving triangular matrix $A \cdot L^{-T}$ & \texttt{trsm} & $mn^2$ \\
\bottomrule
\end{tabular}
\label{table:elem_flops_gpu}
\end{table}

Using the operation costs from Table~\ref{table:elem_flops_gpu}, each iteration of the sequential algorithm requires one \texttt{trsm}, \texttt{syrk}, and \texttt{potrf} operation on $n \times n$ blocks, totaling $\frac{7n^3}{3}$ flops per block. Since the first block requires only the initial Cholesky factorization costing $\frac{n^3}{3}$ flops, the total computational cost becomes
\begin{equation*}
    \left(\frac{7}{3}N - 2 \right)n^3
\end{equation*}
flops. This results in $\mathcal{O}(Nn^3)$ complexity that scales linearly with the number of blocks $N$. However, the sequential dependencies preclude parallelization across blocks, which motivates the matrix reordering strategies developed in subsequent sections to expose parallelism within the factorization process.

\section{Parallelization Strategies} \label{sec:parallel}

The sequential dependency in Algorithm~\ref{alg:seq_factorization} arises from the natural ordering of blocks, where each block must wait for the previous one to complete. However, this dependency can be relaxed by reordering the matrix blocks through carefully chosen permutations. Applying a permutation matrix $P$, we can transform the original problem $\Psi = LL^\top$ into $P\Psi P^\top = \hat{L}\hat{L}^\top$, where the permuted matrix exposes opportunities for parallel computation.

The key insight is to partition blocks into independent sets that can be processed concurrently while preserving the mathematical correctness of the factorization. We present two distinct permutation approaches that offer different computational trade-offs. The first approach targets scenarios with limited parallel units and potentially high synchronization overhead, typical of classical CPUs and distributed computing clusters. We then extend this framework to accommodate devices with massive parallel processing capabilities and low-to-moderate synchronization costs, characteristic of modern GPUs.

\subsection{Partition Permutation} \label{sec:partition_permutation}

The first approach, adapted from \cite{cao2002}, \textit{partitions} the matrix $\Psi$ as follows:
\begin{equation*}
    \Psi = \left[
    \begin{array}{c|c|c|c|c|c|c}
        \begin{matrix}
        D_{11} & \star &  \\
        E_{11} & \ddots & \ddots \\
         & \ddots & D_{N_{1}1} \\
        \end{matrix} & \begin{matrix}\\[0.9cm] \star\end{matrix} & & & & & \\
    \hline
    \hspace{2cm} F_1 & A_2 & \star\hspace{2cm} & & & & \\
    \hline
    & \begin{matrix}
        B_2 \\[1.3cm]
    \end{matrix} & \begin{matrix}
        D_{12} & \star & \\
        E_{12} & \ddots & \ddots \\
         & \ddots & D_{N_{2}2} \\
        \end{matrix} & \begin{matrix}\\[0.9cm] \star\end{matrix} & & & \\
    \hline
    & & \hspace{2cm}F_2 & A_3 & \star\hspace{1.5cm} & & \\
    \hline
    & & & \begin{matrix}
        B_3 \\[1.3cm]
    \end{matrix} & \ddots & \begin{matrix}\\[0.5cm] \star\end{matrix} &  \\
    \hline
    & & & & & \begin{matrix}
        B_p \\[1.3cm]
    \end{matrix} & \begin{matrix}
        D_{1p} & \star & \\
        E_{1p} & \ddots & \ddots  \\
         &  \ddots & D_{N_{p}p}
    \end{matrix}
    \end{array}
    \right],
\end{equation*}
where $p-1$ pivots partition the matrix into $p$ chunks that can be processed in parallel across $p$ threads. A permutation $P_p$ moves the pivots to the end, yielding the permuted matrix
\begin{equation*}
\resizebox{1.0\hsize}{!}{$
    \Psi_p = \left[
    \begin{array}{c|c|c|c||c|c|c|c}
        \begin{matrix}
        D_{11} & \star &  \\
        E_{11} & \ddots & \ddots \\
         & \ddots & D_{N_{1}1} \\
        \end{matrix} & & & & & & & \\
    \hline
     & \begin{matrix}
        D_{12} & \star & \\
        E_{12} & \ddots & \ddots \\
         & \ddots & D_{N_{2}2} \\
        \end{matrix} & & & & & & \\
    \hline
    & & \ddots & & & & & \\
    \hline
    & & & \begin{matrix}
        D_{1p} & \star & \\
        E_{1p} & \ddots & \ddots  \\
         &  \ddots & D_{N_{p}p}
    \end{matrix} & & & & \\
    \hline\hline
    \hspace{2cm} F_{1}\rule{0pt}{2.2ex} & \begin{matrix}B_{2}^{\top} & \quad\quad\quad\quad\;\ \end{matrix} & & & A_{2} & & & \\
    \hline
     & \hspace{2cm} F_2\rule{0pt}{2.2ex} & \begin{matrix}B_{3}^{\top} & \quad\quad\quad \end{matrix} & & & A_{3} & & \\
     \hline
     & & \ddots &  & & & \ddots & \\
     \hline
     & & \hspace{1cm} F_{p-1}\rule{0pt}{2.3ex} & \begin{matrix}B_{p}^\top & \quad\quad\quad\quad\quad \end{matrix} & & & & A_{p}
    \end{array}
    \right].
$}
\end{equation*}
This breaks the sequential dependency in the Cholesky factorization by decoupling the columns of individual partitions through pivot relocation. The resulting Cholesky factors $L_p$ are given by
\begin{equation*}
\resizebox{1.0\hsize}{!}{$
    L_p =
    \left[
    \begin{array}{c|c|c|c||c|c|c|c}
        \begin{matrix}
        \hat{D}_{11} &  &  \\
        \hat{E}_{11} & \ddots &  \\
         & \ddots & \hat{D}_{N_{1}1} \\
        \end{matrix} & & & & & & & \\
    \hline
     & \begin{matrix}
        \hat{D}_{12} &  & \\
        \hat{E}_{12} & \ddots & \\
         & \ddots & \hat{D}_{N_{2}2} \\
        \end{matrix} & & & & & & \\
    \hline
    & & \ddots & & & & & \\
    \hline
    & & & \begin{matrix}
        \hat{D}_{1p} &  & \\
        \hat{E}_{1p} & \ddots &  \\
         &  \ddots & \hat{D}_{N_{p}p}
    \end{matrix} & & & & \\
    \hline \hline
    \hspace{2cm} \hat{F}_{1} & \hat{B}_{12}^\top \hspace{0.3cm} \dots \hspace{0.3cm} \hat{B}_{N_22}^\top & & & \hat{A}_{2} & & & \\
    \hline
     & \hspace{2cm} \hat{F}_2 & \hat{B}_{13}^\top \hspace{0.3cm} \dots \hspace{0.3cm} \hat{B}_{N_33}^\top & & \hat{H}_2 & \hat{A}_{3} & & \\
     \hline
     & & \ddots &  & & \ddots & \ddots & \\
     \hline
     & & \hspace{1.5cm} \hat{F}_{p-1} & \hat{B}_{1p}^\top \hspace{0.3cm} \dots \hspace{0.3cm} \hat{B}_{N_pp}^\top & & & \hat{H}_{p-1} & \hat{A}_{p}
    \end{array}
    \right].
$}
\end{equation*}
Note the additional fill-in entries $\hat{B}_{ik}$ and $\hat{H_k}$ for $i=2,\dots,N_k$ and $k=2,\dots,p$, which arise from the permutation. While fill-in increases the total computational cost, this approach enables parallelization that can outweigh the additional operations.

Algorithm~\ref{alg:factorization_partition} presents the detailed factorization procedure for $\Psi_p$, which operates in two distinct phases: a parallel phase followed by a sequential phase. By moving the pivots to the matrix end, we eliminate dependencies and enable parallelization of each partition, after which a sequential phase proceeds that closely resembles Algorithm~\ref{alg:seq_factorization}.

To illustrate the approach, consider the factorization of the second partition of $\Psi_p$. The factorization of the block tridiagonal submatrix containing elements $D_{12}, \dots$ and $E_{12}, \dots$ follows the sequential Algorithm~\ref{alg:seq_factorization} and corresponds to Lines 5 to 7 in Algorithm~\ref{alg:factorization_partition}. The key difference from the sequential algorithm is that we must also handle the fill-in arising from the pivots, as shown in Lines 8 to 12. Specifically, the introduction of the permutation $B_2^\top$ and $E_{12}$ generates the fill-in $B_{22}^\top$, which then propagates recursively to produce fill-in across the entire row, i.e., $B_{12}^\top, \dots, B_{N_22}^\top$.

We now address the update of the pivots $\hat{A}_i$ and their associated fill-in $\hat{H}_i$. Algorithm~\ref{alg:factorization_partition} is carefully designed to prevent race conditions in these updates. For instance, $\hat{F_i}$ and $\hat{B}_{i,k+1}$ cannot simultaneously update $\hat{A}_{i+1}$ for $i=1,\dots,p-1$, as they are handled by different threads. To address this, we defer the $\hat{F_i}$ update to the sequential phase at Line~23, which constitutes a left-looking operation. Consequently, Algorithm~\ref{alg:factorization_partition} represents a hybrid left-/right-looking approach.

Analyzing Algorithm~\ref{alg:factorization_partition} in detail reveals the following computational costs for each phase:
\begin{center}
\begin{tabular}{lccc}
\toprule
 & Parallel phase ($k=1$) & Parallel phase ($k>1$) & Sequential phase \\
\midrule
Cost (flops) & $\left( \frac{7}{3}N_1 - 1 \right)n^3$ & $\left( \frac{19}{3}N_k - 1 \right)n^3$ & $\left( \frac{10}{3}p - \frac{16}{3} \right)n^3$ \\
\bottomrule
\end{tabular}
\end{center}
The first partition exhibits significantly lower cost per horizon element $N_1$ compared to other partitions, leading to the following optimal allocation strategy:
\begin{proposition}
    For a problem with total $N$ diagonal elements, the optimal ratio between $N_1$ and $N_k$ for $k>1$ is given by
    \begin{equation*}
        \frac{N^\star_1}{N^\star_k} = \frac{19}{7} \approx 2.71,
    \end{equation*}
    with
    \begin{equation*}
        N^\star_1 = \frac{19N-19p+19}{7p+12} \text{ and } N^\star_k = \frac{7N-7p+7}{7p+12}.
    \end{equation*}
\end{proposition}
\begin{proof}
    Since parallel threads are independent in the parallel phase, the total cost should be equally distributed for $k=1$ and $k>1$, i.e., $\left( \frac{7}{3}N_1 - 1 \right)n^3 = \left( \frac{19}{3}N_k - 1 \right)n^3$. Furthermore, the number of total nodes must be consistent, i.e., $N-(p-1)=N_1+(p-1)N_k$. Solving this linear system of equations gives us the results above.
\end{proof}
Note that the optimal ratio is independent of both block size $n$ and partition count $p$. Since $N^\star_1$ and $N^\star_k$ are typically non-integer, we employ a practical rounding strategy: round $N^\star_k$ both down and up, compute the corresponding $N^\star_1 = N - (p-1)N^\star_k$, evaluate the associated costs, and select the configuration with the lowest maximum cost.

\begin{figure}[tbp]
\centering
\includegraphics[width=0.9\textwidth]{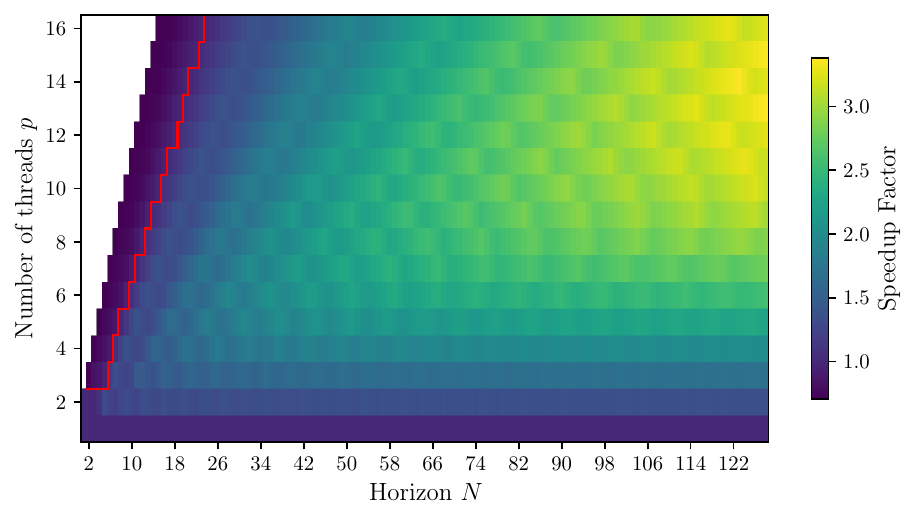}
\caption{Theoretical speed-up of parallel factorization of $\Psi_p$ compared to sequential factorization of $\Psi$, using $p$ number of threads. The red line boundary divides actual speed-ups from slowdowns.}
\label{fig:speedup_partition}
\end{figure}

Figure~\ref{fig:speedup_partition} demonstrates the theoretical speedup achieved by our approach compared to sequential factorization. The speedup increases with both the number of threads $p$ and the horizon length $N$. However, for high thread counts, the fixed overhead of the sequential portion becomes significant, potentially causing slowdowns when the horizon is small. This suggests that limiting the number of threads may be beneficial for short horizons. Additionally, the speedup saturates at a maximum value for any given thread count. As $N \to \infty$, the theoretical maximum speedup converges to
\begin{equation*}  
    \frac{7p}{19}+\frac{12}{19}  
\end{equation*}  
for $p$ threads. This fundamental limit implies that even with two threads, the maximum achievable speedup is only $36.8\%$, while four threads yield a speedup that barely exceeds $2.1$x.

\begin{algorithm}[tbp]
    \caption{Parallel factorization of $\Psi_p$}
    \label{alg:factorization_partition}
    \begin{algorithmic}[1]
        \Statex \textbf{Parallel Phase}
        \For{$k = 1,...,p$} \textbf{in parallel}
            \State \makebox[0pt][l]{$\hat{A}_k \gets A_k$}\hspace{130pt} \textbf{if} $k > 1$
            \State \makebox[0pt][l]{$\hat{B}_{1,k} \gets B_{k}$}\hspace{130pt} \textbf{if} $k > 1$
            \For{$i=1,...,N_p-1$}
                \State $\hat{D}_{ik} \gets \text{chol}(D_{ik})$ \hfill \algPotrf
                \State $\hat{E}_{ik} \gets E_{ik} \hat{D}_{ik}^{-\top}$ \hfill \algTrsm
                \State $\hat{D}_{i+1,k} \gets D_{i+1,k} - \hat{E}_{ik} \hat{E}_{ik}^\top$ \hfill \algSyrk
                \If{$k>1$}
                    \State $\hat{B}_{i,k}^\top \gets \hat{B}_{i,k}^\top \hat{D}_{ik}^{-\top}$ \hfill \algTrsm
                    \State $\hat{A}_k \gets \hat{A}_k - \hat{B}_{ik}^\top \hat{B}_{ik}$  \hfill \algSyrk
                    \State $\hat{B}_{i+1,k}^\top \gets -\hat{B}_{ik}^\top \hat{E}_{ik}^\top$ \hfill \algGemm
                \EndIf
            \EndFor
            \State $\hat{D}_{N_kk} \gets \text{chol}(D_{N_kk})$ \hfill \algPotrf
            \If{$k>1$}
                \State $\hat{B}_{N_kk}^\top \gets \hat{B}_{N_kk}^\top \hat{D}_{N_kk}^{-\top}$  \hfill \algTrsm
                \State $\hat{A}_k \gets \hat{A}_k - \hat{B}_{N_kk}^\top \hat{B}_{N_kk}$ \hfill \algSyrk
            \EndIf
            \State \makebox[0pt][l]{$\hat{F}_k \gets F_k \hat{D}_{N_kk}^{-\top}$}\hspace{130pt} \textbf{if} $k<p$ \hfill \algTrsm
            \State \makebox[0pt][l]{$\hat{H}_k \gets -\hat{F}_k \hat{B}_{N_kk}$}\hspace{130pt} \textbf{if} $1<k<p$ \hfill \algGemm

        \EndFor
        \Statex \textbf{Sequential Phase}
        \For{$k = 2,...,p-1$}
            \State $\hat{A}_k \gets \hat{A}_k - \hat{F}_{k-1} \hat{F}_{k-1}^\top$ \hfill \algSyrk
            \State $\hat{A}_k \gets \text{chol}( \hat{A}_k )$ \hfill \algPotrf
            \State $\hat{H}_k \gets \hat{H}_k \hat{A}_k^{-\top}$ \hfill \algTrsm
            \State $\hat{A}_{k+1} \gets \hat{A}_{k+1} - \hat{H}_k \hat{H}_k^\top$ \hfill \algSyrk
        \EndFor
        \State $\hat{A}_{p} \gets \hat{A}_{p} - \hat{F}_{p-1} \hat{F}_{p-1}^\top$ \hfill \algSyrk
        \State $\hat{A}_{p} \gets \text{chol}( \hat{A}_{p} )$ \hfill \algPotrf
    \end{algorithmic}
\end{algorithm}

We implemented this parallelization strategy in \cite{song2025} for our solver PIQP \cite{schwan2023} as part of the structure exploiting backend \cite{schwan2025}.

\subsection{Single-Stage Permutation}

In this section, we focus on devices such as GPUs that possess massive parallel processing capabilities and efficient synchronization mechanisms. We extend the previous approach to the extreme case where $p \geq \left\lceil N/2 \right\rceil$ parallel processing units are available. We define the permutation $P_1$ such that
\begin{equation*}
\Phi_1 = P_1\Psi P_1^\top = \left[\begin{array}{ccccc|cccc}
D_1 & & & & & & & & \\
& D_3 & & & & & & & \\
& & D_5 & & & & \multicolumn{2}{c}{\star} & \\
& & & D_7 & & & & & \\
& & & & \ddots & & & & \\
\hline
\rule{0pt}{1\normalbaselineskip} E_1 & E_2^\top & & & & D_2 & & & \\
& E_3 & E_4^\top & & & & D_4 & & \\
& & E_5 & E_6^\top & & & & D_6 & \\
& & & \ddots & \ddots & & & & \ddots
\end{array}\right].
\end{equation*}
This reordering creates a block structure in which all odd-indexed diagonal blocks $(D_1, D_3, \ldots)$ appear first, followed by even-indexed blocks $(D_2, D_4, \ldots)$. The critical observation is that all odd-indexed blocks are independent of each other and can be factorized in parallel, as they only depend on the original matrix data. The Cholesky factorization $\Phi_1=\hat{L}_1\hat{L}_1^\top$ is then given by
\begin{equation*}
\hat{L}_1 = \left[\begin{array}{ccccc|cccc}
\hat{D}_1 & & & & & & & & \\
& \hat{D}_3 & & & & & & & \\
& & \hat{D}_5 & & & & & & \\
& & & \hat{D}_7 & & & & & \\
& & & & \ddots & & & & \\
\hline
\rule{0pt}{1\normalbaselineskip} \hat{E}_1 & \hat{E}_2^\top & & & & \hat{D}_2 & & & \\
& \hat{E}_3 & \hat{E}_4^\top & & & \hat{H}_1 & \hat{D}_4 & & \\
& & \hat{E}_5 & \hat{E}_6^\top & & & \hat{H}_2 & \hat{D}_6 & \\
& & & \ddots & \ddots & & & \ddots & \ddots
\end{array}\right]
\end{equation*}
where new fill-in blocks $\hat{H}_i$ emerge as a consequence of the permutation.

To understand why this parallelization is possible, consider column two of $\Phi_1$ as an example. The elements $D_3$ and $E_2^\top$ update $\hat{D_2}$, while $D_3$ and $E_3$ update $\hat{D_4}$, and $E_2^\top$ and $E_3$ produce the fill-in $\hat{H}_1$. This pattern holds consistently for all other columns in $\Phi_1$ with odd-indexed diagonal elements. Crucially, we do not update any values in other columns with odd-indexed diagonal elements during this process. Therefore, the factorization of these columns can be fully parallelized without dependencies.

\begin{figure*}[t]
    \centering
    \begin{subfigure}[t]{0.45\textwidth}
        \centering
        \includegraphics[trim={1cm 0.8cm 1cm 0.8cm},clip,width=\textwidth]{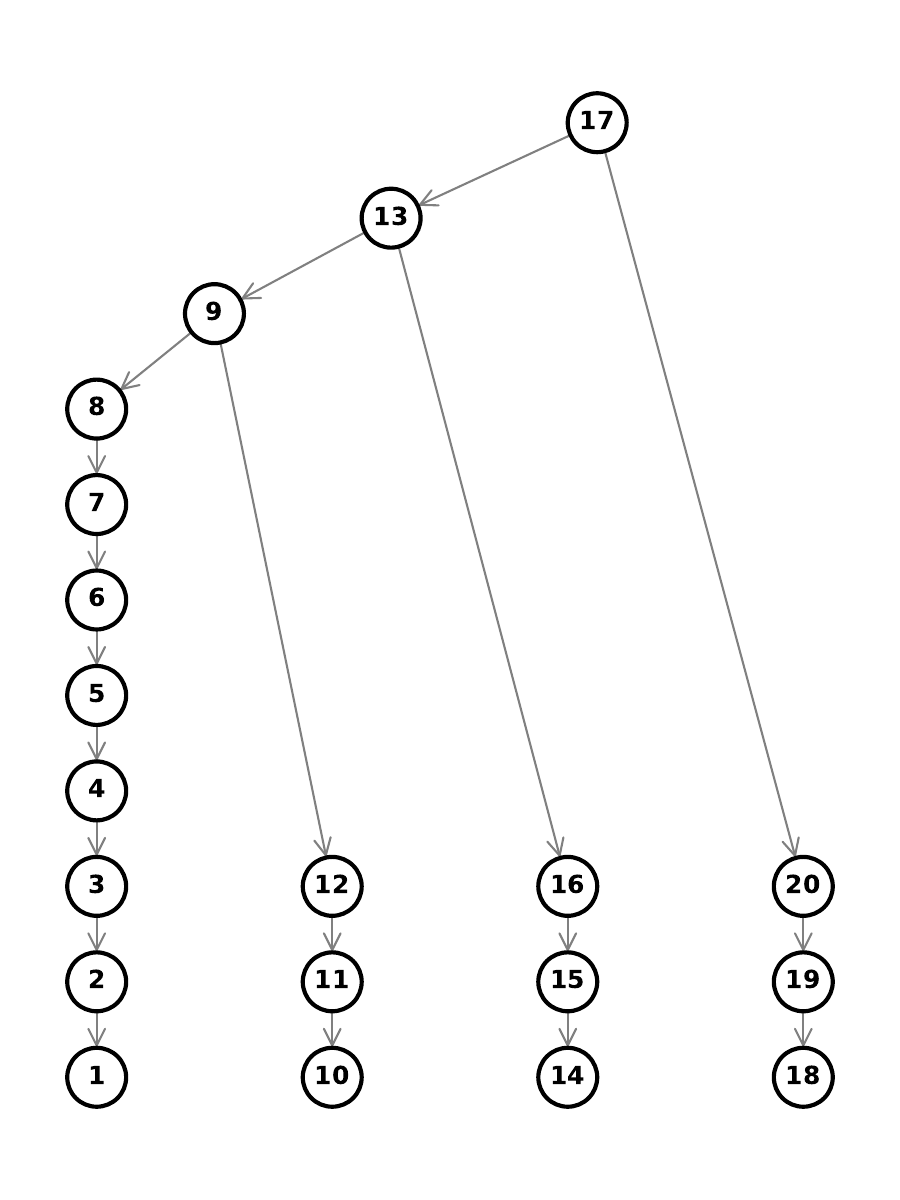}
        \caption{Elimination trees of $\Psi_4$ for $N=20$.}
        \label{fig:etree20_p}
    \end{subfigure}
    \hspace{0.05\textwidth}
    \begin{subfigure}[t]{0.45\textwidth}
        \centering
        \includegraphics[trim={1cm 0.8cm 1cm 0.8cm},clip,width=\textwidth]{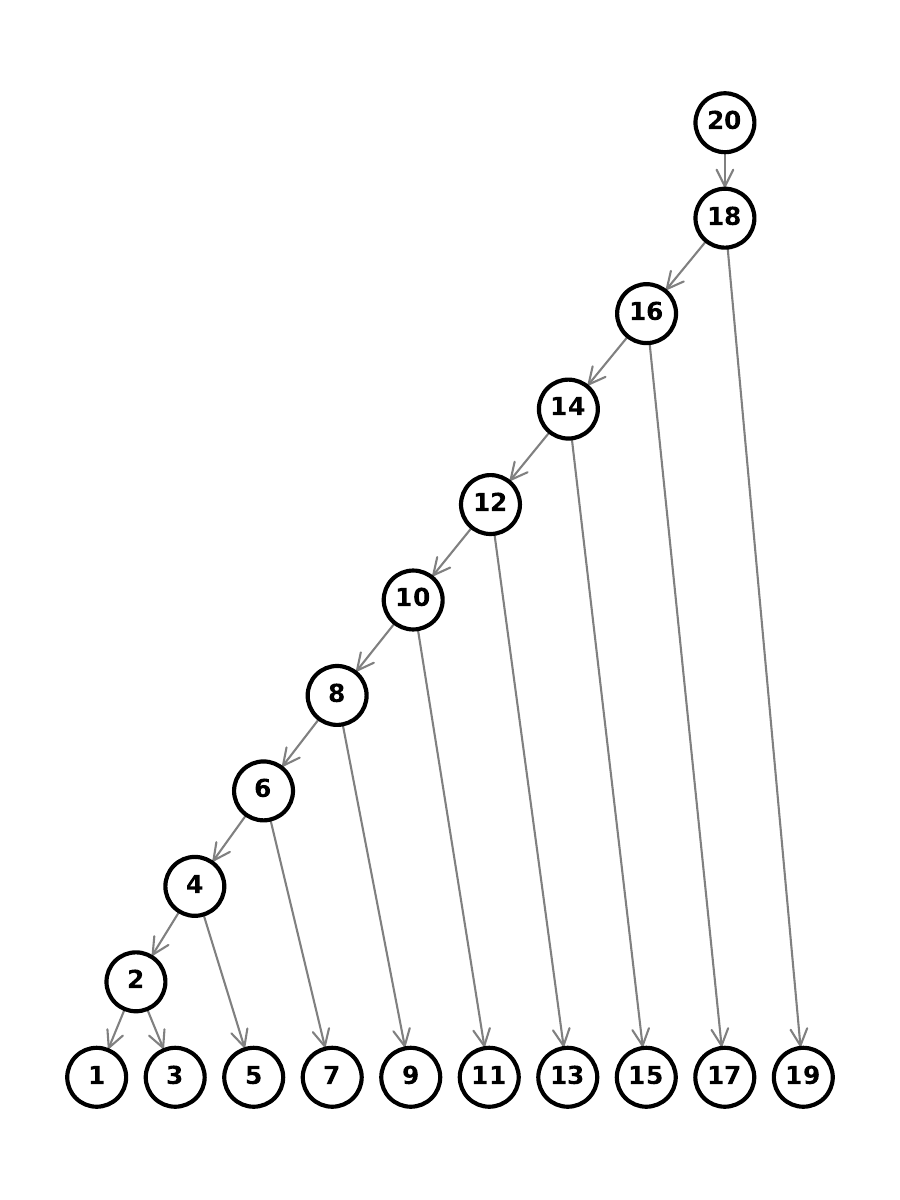}
        \caption{Elimination trees of $\Phi_1$ for $N=20$.}
        \label{fig:etree20_1}
    \end{subfigure}
    \caption{Elimination trees for different permutations of $A$ for $N=20$.}
\end{figure*}

\begin{algorithm}[t]
\caption{Parallel factorization of $\Phi_1$}
\label{alg:factorization_single_stage}
\begin{algorithmic}[1]
\For{$i = 1,3,\dots,N$} \textbf{in parallel} 
    \State $\hat{D}_i \leftarrow \chol(D_i)$ \hfill \algPotrf
    \State \makebox[0pt][l]{$\hat{E}_i \leftarrow E_i \hat{D}_i^{-\top}$}\hspace{130pt} if $i + 1 \leq N$ \hfill \algTrsm
    \State \makebox[0pt][l]{$\hat{D}_{i+1} \leftarrow D_{i+1} - \hat{E}_i \hat{E}_i^\top$}\hspace{130pt} if $i + 1 \leq N$ \hfill \algSyrk
    \State \makebox[0pt][l]{$\hat{E}_{i-1} \leftarrow \hat{D}_{i}^{-1} E_{i-1}$}\hspace{130pt} if $i > 1$ \hfill \algTrsm
    \State \makebox[0pt][l]{$H_{(i-1)/2} \leftarrow -\hat{E}_{i} \hat{E}_{i-1}$}\hspace{130pt} if $i > 1$ and $i+1 \leq N$ \hfill \algGemm
\EndFor
\For{$i = 2,4,\dots,N-1$} \textbf{in parallel}
    \State $\hat{D}_i \leftarrow \hat{D}_i - \hat{E}_i^\top \hat{E}_i$ \hfill \algSyrk
\EndFor
\State $\hat{D}_2 \leftarrow \chol(\hat{D}_2)$ \hfill \algPotrf
\For{$i = 4,6,\dots,N$}
    \State $H_{i/2-1} \leftarrow H_{i/2-1}\hat{D}_{i-2}^{-\top}$ \hfill \algTrsm
    \State $\hat{D}_i \leftarrow \hat{D}_i - H_{i/2-1} H_{i/2-1}^\top$ \hfill \algSyrk
    \State $\hat{D}_i \leftarrow \chol(\hat{D}_i)$ \hfill \algPotrf
\EndFor
\end{algorithmic}
\end{algorithm}

A useful way to visualize the computational dependencies in a Cholesky factorization is through the elimination tree \cite{heath1991}. The elimination tree is a powerful tool for understanding computational dependencies and predicting fill-in patterns in the resulting factorization. A parent node cannot be processed until all its child nodes have been fully processed. Moreover, fill-in is produced at position $(j,k)$ if there exists a path from node $i$ to node $k$ that passes through node $j$ during the elimination of node $i$.

Examining the elimination tree of $\Phi_1$ shown in Figure~\ref{fig:etree20_1}, we observe the independence of the odd-indexed diagonal blocks. The lowest level exhibits no dependencies and can therefore be processed in parallel. The elimination tree of $\Psi_4$ similarly reveals the parallel processing structure of the four partitions. However, it also exposes a limitation of the permutation $\Phi_1$; while the first level can be fully parallelized, the remaining nodes must be processed sequentially. This observation suggests that minimizing the tree height is desirable, as it determines the minimum number of sequential operations required. We address this issue in the next section.

Algorithm~\ref{alg:factorization_single_stage} presents the optimized parallel version of the factorization of $\Phi_1$, which constitutes a special case of Algorithm~\ref{alg:factorization_partition}. The computational cost with $p$ parallel threads is
\begin{equation*}
    \left(\left\lceil\frac{\left\lceil N/2\right\rceil}{p}\right\rceil\frac{19}{3} + \left\lfloor\frac{N}{2}\right\rfloor \frac{7}{3} - 2 \right) n^3
\end{equation*}
flops for $N \geq 3$. Figure~\ref{fig:speedup_one_level} illustrates the speedup relative to the sequential algorithm. Although substantial parallelism is achieved, the sequential nature of the second phase constrains overall performance. Additionally, the fill-in introduces considerable overhead when not executed in parallel. As observed in Figure~\ref{fig:speedup_one_level}, speedup is only realized with more than 3-4 threads due to this overhead.

\begin{figure}[tbp]
\centering
\vspace{-4mm}
\includegraphics[width=0.9\textwidth]{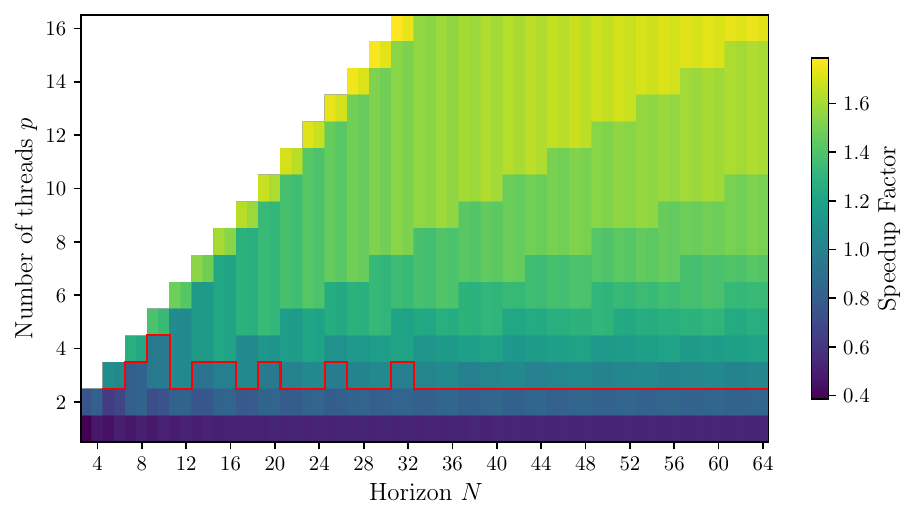}
\caption{Speed-up of parallel factorization of $\Phi_1$ compared to sequential factorization of $\Psi$. The red line boundary divides actual speed-ups from slowdowns.}
\label{fig:speedup_one_level}
\end{figure}

\subsection{Multi-Stage Permutation} \label{seq:mult_stage_perm}

The primary limitation of the single-stage permutation is the lengthy sequential tail that cannot be parallelized. The solution is to recursively permute the sequential portion until no elements remain. This approach results in the same as applying the nested dissection permutation \cite{george1978}, which has proven highly effective in parallel sparse factorization methods \cite{heath1991}. Applying this recursive strategy to $\Psi$ yields the permuted matrix
\begin{equation*}
\Phi_\infty = P_\infty \Psi P_\infty^\top = \left[\begin{array}{ccccc|cccc|ccc|c}
D_1 & & & & & & & & & & & & \\
& D_3 & & & & & & & & & & & \\
& & D_5 & & & \multicolumn{4}{c|}{\star} & & \star & & \star \\
& & & D_7 & & & & & & & & & \\
& & & & \ddots & & & & & & & & \\
\hline
\rule{0pt}{1\normalbaselineskip} E_1 & E_2^\top & & & & D_2 & & & & & & & \\
& & E_5 & E_6^\top & & & D_6 & & & & & & \\
& & & & E_9 & & & D_{10} & & & & & \\
& & & & \ddots & & & & \ddots & & & & \\
\hline
\rule{0pt}{1\normalbaselineskip} & E_3 & E_4^\top & & & & & & & D_4 & & & \\
& & & & & & & & & & D_{12} & & \\
& & & & \ddots & & & & & & & \ddots \\
\hline
\rule{0pt}{1\normalbaselineskip} & & & & \ddots & & & & & & & & \ddots
\end{array}\right],
\end{equation*}
with corresponding Cholesky factor
\begin{equation*}
\hat{L}_\infty = \left[\begin{array}{ccccc|cccc|ccc|c}
\hat{D}_1 & & & & & & & & & & & & \\
& \hat{D}_3 & & & & & & & & & & & \\
& & \hat{D}_5 & & & & & & & & & & \\
& & & \hat{D}_7 & & & & & & & & & \\
& & & & \ddots & & & & & & & & \\
\hline
\rule{0pt}{1\normalbaselineskip} \hat{E}_{1,1} & \hat{E}_{1,2}^\top & & & & \hat{D}_2 & & & & & & & \\
& & \hat{E}_{1,5} & \hat{E}_{1,6}^\top & & & \hat{D}_6 & & & & & & \\
& & & & \hat{E}_{1,9} & & & \hat{D}_{10} & & & & & \\
& & & & \ddots & & & & \ddots & & & & \\
\hline
\rule{0pt}{1\normalbaselineskip} & \hat{E}_{1,3} & \hat{E}_{1,4}^\top & & & \hat{E}_{2,1} & \hat{E}_{2,2}^\top & & & \hat{D}_4 & & & \\
& & & & & & & \hat{E}_{2,5} & \hat{E}_{2,6} & & \hat{D}_{12} & & \\
& & & & \ddots & & & & \ddots & & & \ddots \\
\hline
\rule{0pt}{1\normalbaselineskip} & & & & \ddots & & & & \ddots & & & \ddots & \ddots
\end{array}\right].
\end{equation*}
While the resulting permutation is identical, we have the advantage of knowing the structure beforehand, which allows us to optimize the memory layout and access patterns. This can significantly improve performance, as we demonstrate in the later sections.

Note that the same permutation implicitly arises when applying (block) cyclic reduction \cite{heller1976, gander1998} to the original block tridiagonal matrix. While equivalent, the perspective of computing the Cholesky factorization of a permuted matrix appears strictly more general and enables natural extensions to arrow structures, similar to \cite{song2025}.

Algorithm~\ref{alg:factorization_multi_stage} provides the detailed factorization procedure for $\Phi_\infty$, with a visual representation shown in Figure~\ref{fig:fac_viz}. At each iteration, one level is processed by parallel threads. 

To illustrate the algorithm, we walk through the example outlined in Figure~\ref{fig:fac_viz}, which processes the second column of the second iteration. The operation performing the Cholesky factorization $\hat{D}_6 \gets \chol(\hat{D}_6)$, followed by the triangular solves $\hat{E}_{2,3} \gets \hat{E}_{2,3}\hat{D}_6^{-\top}$ and $\hat{E}_{2,2} \gets \hat{D}_6^{-1}\hat{E}_{2,2}$ (cdiv), and the fill-in producing update $\hat{E}_{3,1} \gets -\hat{E}_{2,3}\hat{E}_{2,2}^\top$, are identical to those in the single-stage base algorithm.
However, we must take care to avoid introducing race conditions. In a purely right-looking algorithm, we would update both $\hat{D}_4 \gets \hat{D}_4 - \hat{E}_{2,2}^\top\hat{E}_{2,2}$ and $\hat{D}_8 \gets \hat{D}_8 - \hat{E}_{2,3}\hat{E}_{2,3}^\top$. Unfortunately, the thread processing the first column also updates $\hat{D}_4 \gets \hat{D}_4 - \hat{E}_{2,1}\hat{E}_{2,1}^\top$, resulting in a race condition. To avoid this, we defer the operation $\hat{D}_4 \gets \hat{D}_4 - \hat{E}_{2,2}^\top\hat{E}_{2,2}$ to the next iteration, as symbolized by the dashed arrow in Figure~\ref{fig:fac_viz}. In the subsequent iteration, this becomes a left-looking operation. In our particular example, these correspond to the operations $\hat{D}_6 \gets \hat{E}_{1,6}^\top\hat{E}_{1,6}$ and $\hat{D}_8 \gets \hat{E}_{1,8}^\top\hat{E}_{1,8}$. Thus, this algorithm becomes a hybrid left-/right-looking algorithm.

\begin{algorithm}[t]
\caption{Factorization of $\Phi_\infty$ optimized for parallel}
\label{alg:factorization_multi_stage}
\begin{algorithmic}[1]
\For{$i = 1,\dots,N$} \textbf{in parallel} 
    \State $\hat{D}_i \leftarrow D_i$
    \State \makebox[0pt][l]{$\hat{E}_{1,i} \leftarrow E_{i}$}\hspace{170pt} if $i < N$
\EndFor
\For{$s = 1,2,4,\dots,2^{\lfloor\log_2(N)\rfloor}$}
\For{$i = s,3s,5s,\dots,N$} \textbf{in parallel} 
    \State \makebox[0pt][l]{$\hat{D}_i \leftarrow \hat{D}_i - \hat{E}_{s-1,2i/s}^\top \hat{E}_{s-1,2i/s}$}\hspace{155pt} if $s > 1$ and $i \leq N - s/2$ \hfill \algSyrk
    \State $\hat{D}_i \leftarrow \chol(\hat{D}_i)$ \hfill \algPotrf
    \State \makebox[0pt][l]{$\hat{D}_{i+s} \leftarrow \hat{D}_{i+s}$}\hspace{155pt} if $s > 1$ and $i + s \leq N - s/2$ \hfill \algSyrk
    \Statex $\hspace{6.5em} - \hat{E}_{s-1,2i/s+2}^\top \hat{E}_{s-1,2i/s+2}$
    \State \makebox[0pt][l]{$\hat{E}_{s,i/s} \leftarrow \hat{E}_{s,i/s} \hat{D}_i^{-\top}$}\hspace{155pt} if $i + s \leq N$ \hfill \algTrsm
    \State \makebox[0pt][l]{$\hat{D}_{i+s} \leftarrow \hat{D}_{i+s} - \hat{E}_{s,i/s} \hat{E}_{s,i/s}^\top$}\hspace{155pt} if $i + s \leq N$ \hfill \algSyrk
    \State \makebox[0pt][l]{$\hat{E}_{s,i/s-1} \leftarrow \hat{D}_{i}^{-1} \hat{E}_{s,i/s-1}$}\hspace{155pt} if $i > s$ \hfill \algTrsm
    \State \makebox[0pt][l]{$\hat{E}_{s+1,(i-s)/(2s)} \leftarrow - \hat{E}_{s,i/s} \hat{E}_{s,i/s-1}$}\hspace{155pt} if $i > s$ and $i + s \leq N$ \hfill \algGemm
\EndFor
\EndFor
\end{algorithmic}
\end{algorithm}

\begin{figure}[tbp]
\centering
\includegraphics[width=1.0\textwidth]{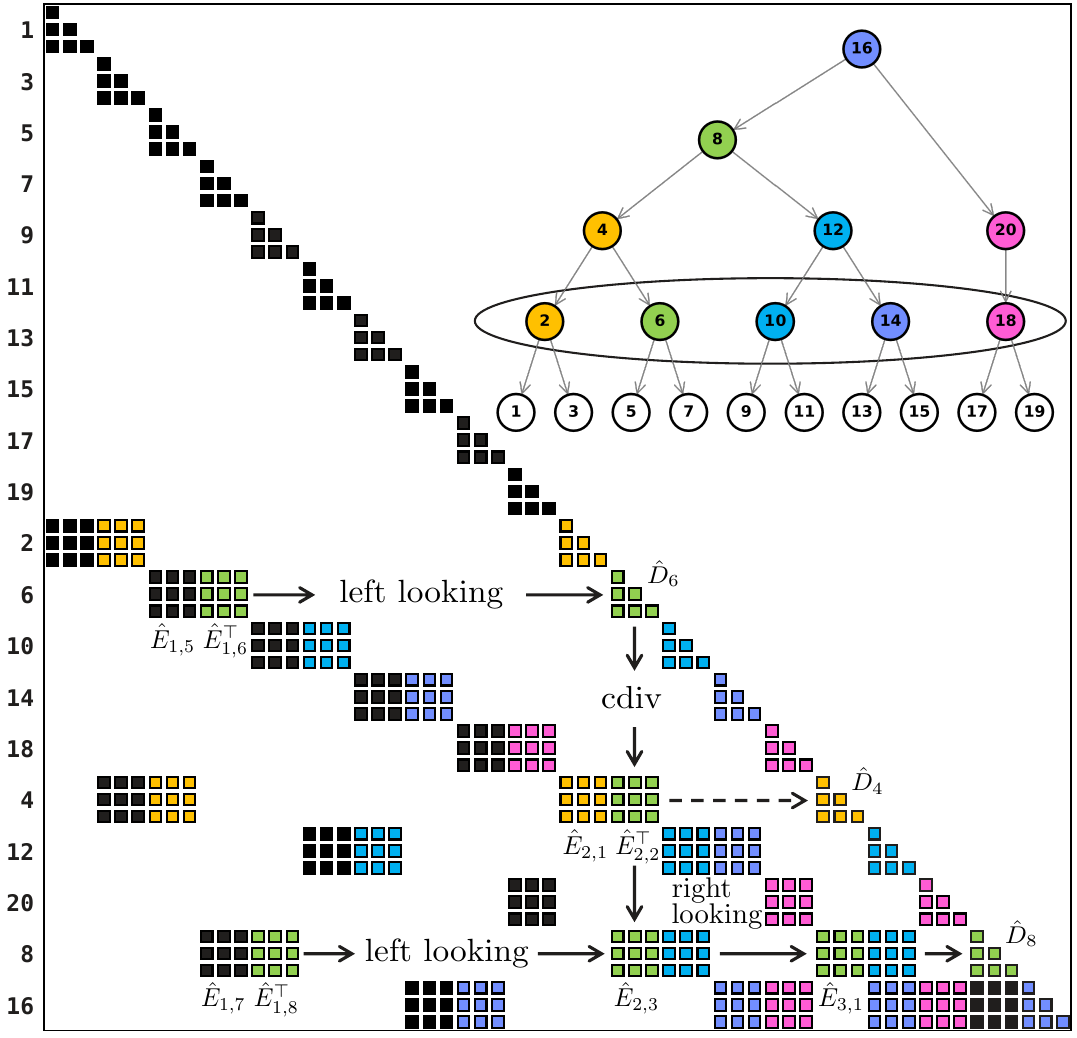}
\caption{Visualization of the Cholesky factorization of $\Phi_\infty$ for $N=20$ using Algorithm~\ref{alg:factorization_multi_stage}. Visualized is the parallel iteration for $s=2$, i.e., the second iteration. The full arrows operations updating data, while the dashed arrow shows a deferred operation.}
\label{fig:fac_viz}
\end{figure}

The algorithm requires $\lfloor\log_2(N)\rfloor+1$ iterations, each of which can be executed in parallel. With $p$ threads available, the computational cost becomes
\begin{equation}
    \left(\left\lceil\frac{N/2}{p}\right\rceil \frac{16}{3} + \sum_{i = 1}^{\lfloor\log_2(N)\rfloor-1}
    \left\lceil\frac{\left\lceil N/2^{i+1}\right\rceil}{p}\right\rceil\frac{22}{3} + \frac{4}{3} \right)n^3
\end{equation}
flops for $N \geq 2$. The first term accounts for the initial iteration, which contains no left-looking operations. The final term represents the last iteration, which consists of a single left-looking operation and one Cholesky factorization. When $p<N/2$, the first term satisfies $\left\lceil\frac{N/2}{p}\right\rceil > 1$, indicating that the available threads are insufficient to fully parallelize the first level. This results in increased execution time for the initial iteration and, consequently, higher overall computational cost. While modern GPUs provide a massive number of parallel units, for sufficiently large $N$, we will eventually saturate them, making this term significant.

\begin{figure}[tbp]
\centering
\includegraphics[width=1.0\textwidth]{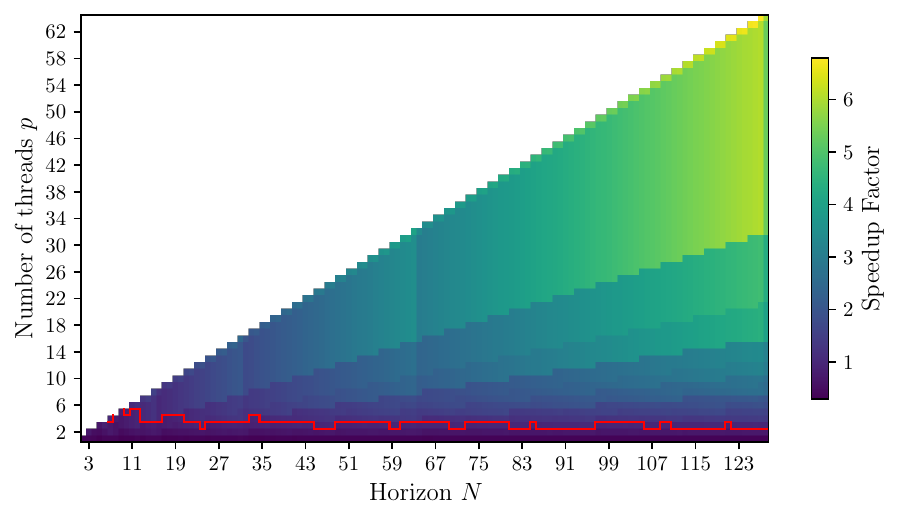}
\caption{Theoretical speed-up of parallel factorization of $\Phi_\infty$ compared to sequential factorization of $\Psi$. The red line boundary divides actual speed-ups from slowdowns.}
\label{fig:speedup_all_level}
\end{figure}

This approach reduces the complexity from $\mathcal{O}(Nn^3)$ in the sequential case to $\mathcal{O}(\log_2(N)n^3)$ for the parallel multi-stage factorization. Figure~\ref{fig:speedup_all_level} demonstrates the expected speedup relative to sequential factorization. The plot exhibits discrete breakpoints at $\lfloor\log_2(N)\rfloor$ with linear speedup increases between them. This behavior stems from the elimination tree height remaining constant for $N \in [2^i,2^{i+1})$, while the cost per iteration remains unchanged given sufficient parallel cores.

\begin{figure}[tbp]
\centering
\includegraphics[width=1.0\textwidth]{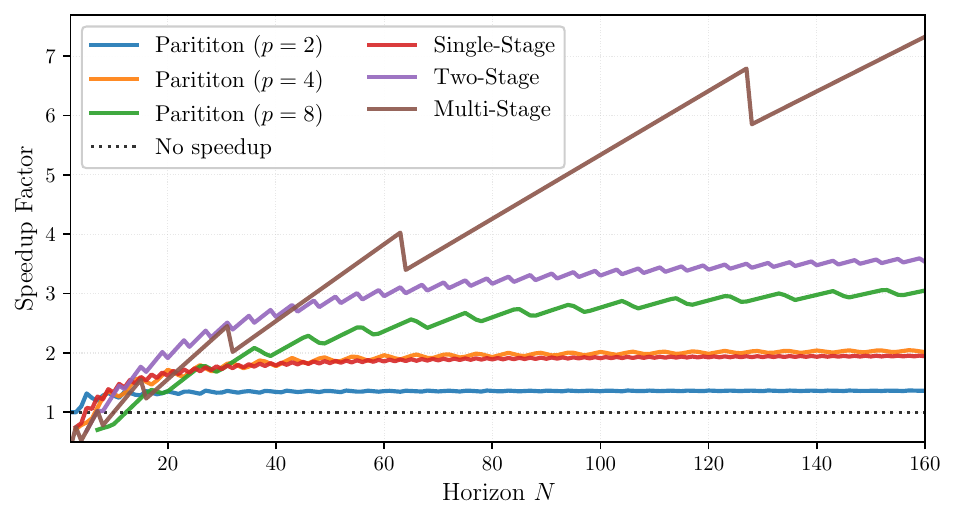}
\caption{Speed-up of parallel factorization of all discussed parallel factorization approaches compared to sequential factorization of $\Psi$. The single-, two-, and multi-stage approaches assume $p=\left\lceil N/2\right\rceil$ parallel threads.}
\label{fig:speedup_comparison_all_strategies}
\end{figure}

Figure~\ref{fig:speedup_comparison_all_strategies} compares the speedup performance of all parallel factorization approaches discussed thus far, relative to the sequential factorization of $\Psi$. We also include the expected speedup for a two-stage variant, which applies the idea underlying the multi-stage permutation but terminates after the second recursive iteration in the derivation. While we do not discuss this variant in detail here, it provides a useful intermediate benchmark. We observe that depending on the horizon length $N$, different methods achieve varying speedups, with some even performing slower than the sequential factorization scheme. This variation stems from the different amounts of overhead and fill-in produced by each method. Consequently, one could envision a dispatch mechanism that selects the method yielding the best speedup for a given problem instance. In the remainder of this chapter, we focus on the multi-stage variant as it provides the largest speedup for large horizons $N$, leaving such adaptive method selection as future work.

\subsubsection{Cholesky Solve}

\begin{algorithm}[t]
\caption{Cholesky solve $\Psi x = P_\infty^\top \hat{L}_\infty \hat{L}_\infty^\top P_\infty x = b$ optimized for parallel}
\label{alg:solve_multi_stage}
\begin{algorithmic}[1]
\Statex \textbf{Forward substitution: $P_\infty^\top \hat{L}_\infty y = b$}
\State $y \leftarrow b$
\For{$s = 1,2,4,\dots,2^{\lfloor\log_2(N)\rfloor}$}
\For{$i = s,3s,5s,\dots,N$} \textbf{in parallel} 
    \State $y_i \leftarrow \hat{D}_i^{-1} y_i$ \hfill \texttt{trsm} \hspace{5pt} $n^2 m$
    \State \makebox[0pt][l]{$y_{i+s} \leftarrow y_{i+s} - \hat{E}_{s,i/s} y_i$}\hspace{150pt} if $i + s \leq N$ \hfill \texttt{gemm} \hspace{0pt} $2n^2 m$
    \State \makebox[0pt][l]{$y_{i-s} \leftarrow y_{i-s} - \hat{E}_{s,i/s-1}^\top y_i$}\hspace{150pt} if $i > s$ \hfill \texttt{gemm} \hspace{0pt} $2n^2 m$
\EndFor
\EndFor
\Statex \textbf{Backward substitution: $\hat{L}_\infty^\top P_\infty x = y$}
\State $x \leftarrow y$
\For{$s = 2^{\lfloor\log_2(N)\rfloor}, 2^{\lfloor\log_2(N)\rfloor-1}, \dots, 2, 1$}
\For{$i = s,3s,5s,\dots,N$} \textbf{in parallel} 
    \State \makebox[0pt][l]{$x_i \leftarrow \hat{x}_i - \hat{E}_{s,i/s}^\top x_{i+s}$}\hspace{150pt} if $i + s \leq N$ \hfill \texttt{gemm} \hspace{0pt} $2n^2 m$
    \State \makebox[0pt][l]{$x_i \leftarrow x_i - \hat{E}_{s,i/s-1} x_{i-s}$}\hspace{150pt} if $i > s$ \hfill \texttt{gemm} \hspace{0pt} $2n^2 m$
    \State $x_i \leftarrow \hat{D}_i^{-\top} x_i$ \hfill \texttt{trsm} \hspace{5pt} $n^2 m$
\EndFor
\EndFor
\end{algorithmic}
\end{algorithm}

So far, we have only examined the factorization $P_\infty \Psi P_\infty^\top = \hat{L}_\infty \hat{L}_\infty^\top$. To solve for a right-hand side, we rewrite $\Psi x = b$ as $\Psi x = P_\infty^\top \hat{L}_\infty \hat{L}_\infty^\top P_\infty x = b$ for $x,b \in \mathbb{R}^{Nn \times m}$. Following the standard Cholesky approach, we find $x$ through two phases: forward substitution followed by backward substitution. In the forward substitution phase, we solve $P_\infty^\top \hat{L}_\infty y = b$ for the auxiliary variable $y \in \mathbb{R}^{Nn\times m}$, then in the backward phase we solve $\hat{L}_\infty^\top P_\infty x = y$. Algorithm~\ref{alg:solve_multi_stage} presents this procedure. The solution phase requires $2\lfloor\log_2(N)\rfloor + 2$ iterations and with $p$ threads, the computational cost is
\begin{equation}
    \left(\sum_{i = 0}^{\lfloor\log_2(N)\rfloor}
    \left\lceil\frac{\left\lceil N/2^{i+1}\right\rceil}{p}\right\rceil10 - 8\left( \left\lceil\frac{N/2}{p}\right\rceil  + 1 \right) \right)n^2m
\end{equation}
flops. The final term accounts for the reduced number of matrix operations when $s=1$ and $s=2^{\lfloor\log_2(N)\rfloor}$. This yields $\mathcal{O}(\log_2(N)n^2m)$ complexity compared to $\mathcal{O}(Nn^2m)$ for the sequential variant.

\subsection{Multi-Stage Permutation with Operation Parallelism} \label{sec:multi_stage_streams}

\begin{algorithm}[t]
\caption{Right Looking Factorization of $\Phi_\infty$ optimized for parallel}
\label{alg:factorization_multi_stage_right_looking_atomics}
\begin{algorithmic}[1]
\For{$i = 1,\dots,N$} \textbf{in parallel} 
    \State $\hat{D}_i \leftarrow D_i$
    \State \makebox[0pt][l]{$\hat{E}_{1,i} \leftarrow E_{i}$}\hspace{165pt} if $i < N$
\EndFor
\For{$s = 1,2,4,\dots,2^{\lfloor\log_2(N)\rfloor}$}
\For{$i = s,3s,5s,\dots,N$} \textbf{in parallel} 
    \State $\hat{D}_i \leftarrow \chol(\hat{D}_i)$ \hfill \algPotrf
    \State \makebox[0pt][l]{$\hat{E}_{s,i/s-1} \leftarrow \hat{D}_{i}^{-1} \hat{E}_{s,i/s-1}$}\hspace{150pt} if $i > s$ \hfill \algTrsm
    \State \makebox[0pt][l]{$\hat{E}_{s,i/s} \leftarrow \hat{E}_{s,i/s} \hat{D}_i^{-\top}$}\hspace{150pt} if $i + s \leq N$ \hfill \algTrsm
    \State \makebox[0pt][l]{$\hat{D}_{i-s} \xleftarrow{\odot} \hat{D}_{i-s} - \hat{E}_{s,i/s-1}^\top \hat{E}_{s,i/s-1}$}\hspace{150pt} if $i > s$ \hfill \algSyrk
    \State \makebox[0pt][l]{$\hat{D}_{i+s} \xleftarrow{\odot} \hat{D}_{i+s} - \hat{E}_{s,i/s} \hat{E}_{s,i/s}^\top$}\hspace{150pt} if $i + s \leq N$ \hfill \algSyrk
    \State \makebox[0pt][l]{$\hat{E}_{s+1,(i-s)/(2s)} \leftarrow - \hat{E}_{s,i/s} \hat{E}_{s,i/s-1}$}\hspace{150pt} if $i > s$ and $i + s \leq N$ \hfill \algGemm
\EndFor
\EndFor
\end{algorithmic}
\end{algorithm}

Our approach thus far has focused on parallelization at the column level, as illustrated in Figure~\ref{fig:fac_viz}. However, we can exploit additional parallelism at the operation level. For instance, in Algorithm~\ref{alg:factorization_multi_stage}, the left-looking operations on Lines 7 and 9 can execute concurrently. When we trace the operation dependencies, the critical path contains four sequential operations: left-looking rank-$k$ updates (Lines 7 and 9), Cholesky factorization (Line 8), tridiagonal solve (Lines 10 and 12), and right-looking updates (Lines 11 and 13). By introducing atomic operations, we can reduce this critical path to just three operations using a fully right-looking factorization routine, as presented in Algorithm~\ref{alg:factorization_multi_stage_right_looking_atomics}, where atomic operations are denoted with the $\odot$ symbol.

\begin{figure}[tbp]
\centering
\includegraphics[width=1.0\textwidth]{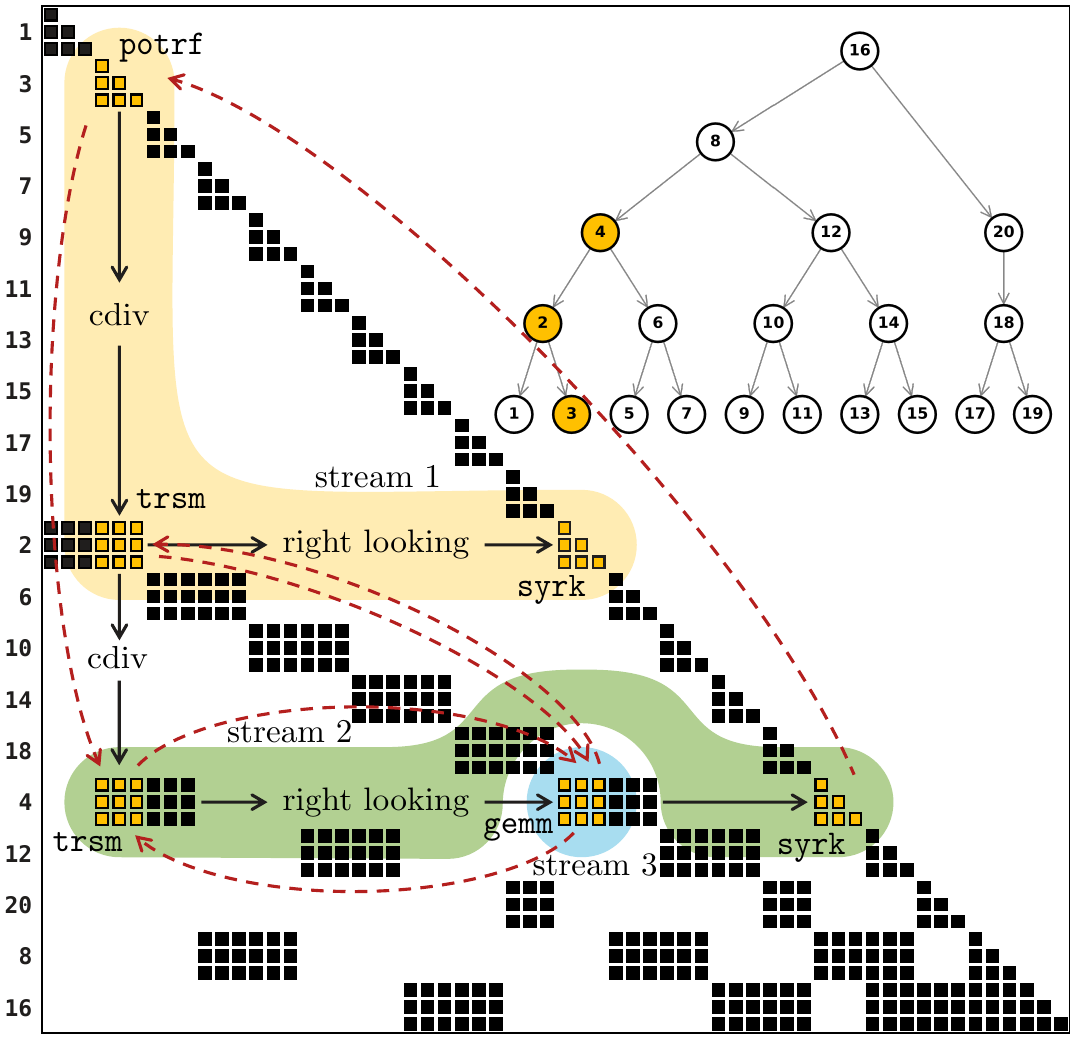}
\caption{Visualization of the Cholesky factorization of $\Phi_\infty$ for $N=20$ using Algorithm~\ref{alg:factorization_multi_stage_right_looking_atomics}. Visualized is the assignment of operations on different CUDA streams that run in parallel. The red dashed arrows represent a synchronization between streams, e.g., the \texttt{gemm} operations in stream 3 has to wait for the \texttt{trsm} operations in stream 1 and 2 to finish first.}
\label{fig:stream_viz}
\end{figure}

The performance overhead of atomic operations on GPUs is minimal under low contention scenarios, with latency comparable to standard memory operations \cite{jia2019dissectingnvidiaturingt4}. Contention is the number of threads that write to the same memory address at the same time. Since our implementation has a maximum contention of only two, the atomic operation cost becomes negligible compared to the benefit of reducing the dependency chain length from four to three operations. We implement this operation-level parallelism using CUDA streams, as detailed in Section~\ref{sec:numerical_implementation_gpu}. Figure~\ref{fig:stream_viz} visualizes how Algorithm~\ref{alg:factorization_multi_stage} maps to CUDA streams, where each highlighted region corresponds to a specific stream. Operations within a single stream execute sequentially, while operations across different streams run in parallel. Inter-stream synchronization ensures that dependency constraints are satisfied. 

We now examine the example illustrated in Figure~\ref{fig:stream_viz} in detail. The computation begins with the Cholesky factorization (\texttt{potrf}), followed by the triangular solves (\texttt{trsm}), which must execute after the Cholesky factorization completes. In stream 1, this ordering is automatically satisfied through sequential execution. However, the triangular solve in stream 2 requires explicit synchronization; otherwise, the \texttt{trsm} operation could execute concurrently with the \texttt{potrf} operation in stream 1. This dependency and its corresponding synchronization event are represented by the red dashed arrow in Figure~\ref{fig:stream_viz}. Similarly, the \texttt{gemm} operation depends on the completion of the preceding \texttt{trsm} operations, but can execute in parallel with the rank-$k$ update (\texttt{syrk}) operations. When the next iteration of Algorithm~\ref{alg:factorization_multi_stage_right_looking_atomics} is scheduled, the \texttt{syrk} operation in stream 2 may still be executing. Consequently, the \texttt{potrf} operation in the subsequent iteration must synchronize accordingly. The same synchronization requirement applies to the \texttt{trsm} operation in stream 2, which must wait for the previous \texttt{gemm} operation to complete.

\begin{sidewaysfigure}[htbp]
\centering
\ifisThesis
\vspace{16cm}
\fi
\includegraphics[width=1\textwidth]{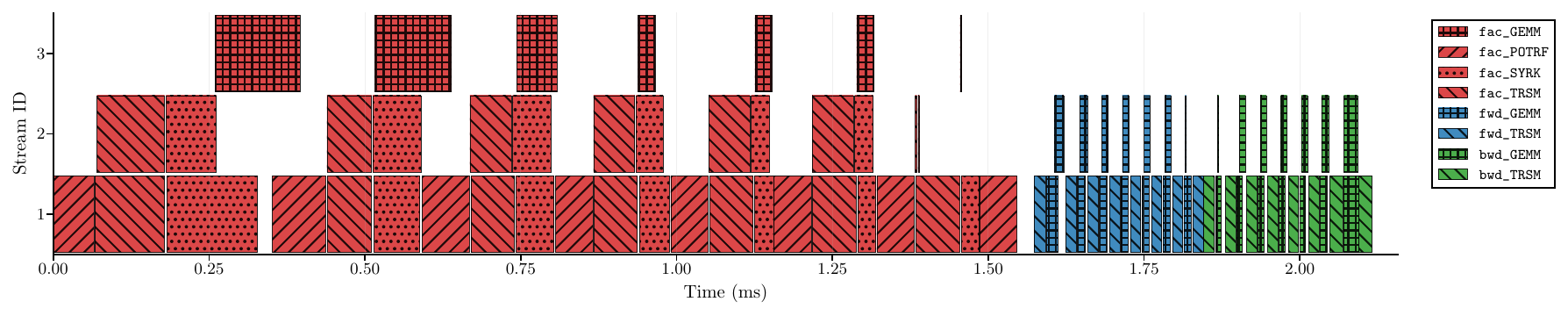} \\
\vspace{3mm}
\includegraphics[width=1\textwidth]{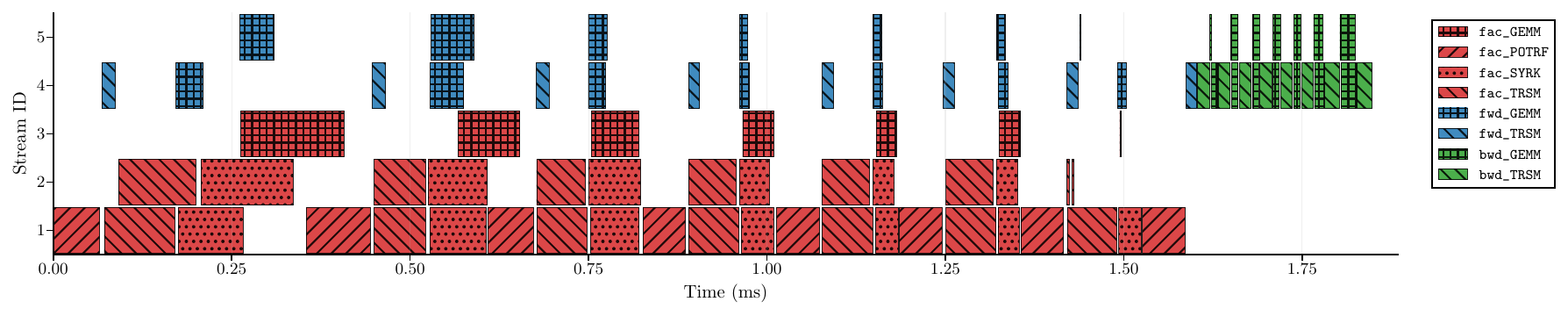}
\caption{Traces of kernel launches of Algorithm~\ref{alg:factorization_multi_stage_right_looking_atomics} (\texttt{fac}: factorization) and Algorithm~\ref{alg:solve_multi_stage} (\texttt{fwd}: forward substitution, \texttt{bwd}: backward substitution) for $n=64$ and $N=128$. In the upper trace, the factorization and solve are run in sequential and in the lower trace, the forward substitution is interlaced with the factorization in parallel.}
\label{fig:fac_sol_timeline_both}
\end{sidewaysfigure}

Figure~\ref{fig:fac_sol_timeline_both} shows traces of kernel launches for Algorithm~\ref{alg:factorization_multi_stage_right_looking_atomics} applied to a problem with $n=64$ and $N=128$. The upper trace clearly demonstrates how different operations overlap and execute in parallel. In particular, note that the first \texttt{gemm} operation overlaps with the \texttt{potrf} of the second iteration, thereby maximizing utilization of available GPU resources. Furthermore, we can execute the forward substitution pass in Algorithm~\ref{alg:solve_multi_stage} concurrently with the factorization as soon as the relevant data becomes available. This is visualized in the lower trace of Figure~\ref{fig:fac_sol_timeline_both}. This overlapping strategy reduces the overall computation time by approximately $12\%$, leading to more efficient GPU resource utilization.

Further implementation details are discussed in Section~\ref{sec:fused_and_blocked_kernels}.

\section{Numerical Implementation} \label{sec:numerical_implementation_gpu}

We conducted all experiments on NVIDIA hardware, specifically using RTX 3080 and RTX 5090 GPUs, while CPU-based experiments were performed on an AMD Ryzen 9 3900X 3.80 GHz processor. The GPU implementation was developed in Python using the Warp library \cite{warp2022}, and the CPU implementation in C++ using the BLASFEO library \cite{frison2018} with custom Python interfaces. The Warp library generates CUDA and C++ code that is just-in-time (JIT) compiled and exposed through a Python interface. This generated code can serve as a foundation for developing more optimized, manually tuned kernels in future work. In particular, more careful memory management could reduce shared memory usage, thereby increasing kernel occupancy.

Before we discuss our numerical results further, we provide a brief introduction to the CUDA programming model for NVIDIA GPUs.

\subsection{CUDA Programming Model}

\begin{figure}[tbp]  
\centering  
\includegraphics[width=1.0\textwidth]{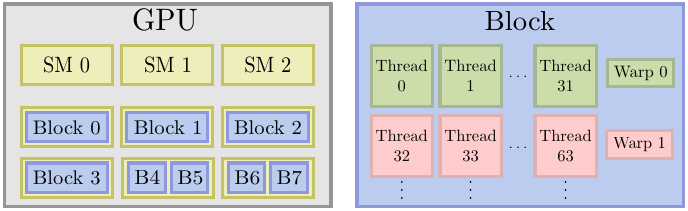}  
\caption{CUDA Programming Model}  
\label{fig:gpu_model}  
\end{figure}

NVIDIA GPUs employ a hierarchical parallel execution model designed for massive parallelism. As illustrated in Figure~\ref{fig:gpu_model}, the CUDA programming model organizes computation into a three-level hierarchy: grids, blocks, and threads.

At the hardware level, a GPU consists of multiple Streaming Multiprocessors (SMs), each capable of executing instructions independently (left panel of Figure~\ref{fig:gpu_model}). When a CUDA kernel is launched, it creates a grid of thread blocks that are distributed across available SMs for execution. Each block is assigned to a single SM and cannot be split across multiple SMs, though one SM may execute multiple blocks concurrently depending on resource availability. It is also possible to run different kernels in parallel as long as there are enough free SMs available, as illustrated with blocks 4-7 in Figure~\ref{fig:gpu_model}.

Within each block, threads are organized into groups of 32 called warps, which represent the fundamental unit of execution on NVIDIA GPUs (right panel of Figure~\ref{fig:gpu_model}). All threads within a warp execute the same instruction simultaneously in a Single Instruction, Multiple Thread (SIMT) fashion. This execution model achieves high throughput when threads in a warp follow the same control flow path, but suffers from performance degradation when threads diverge, as divergent branches must be executed serially.

Thread blocks provide a mechanism for cooperation and synchronization among threads through shared memory. Shared memory is a fast, on-chip memory space accessible to all threads within the same block. This shared memory hierarchy, combined with the ability to synchronize threads within a block, enables efficient implementation of algorithms requiring local data sharing and coordination. However, synchronization across different blocks is generally not supported within a single kernel execution, requiring careful algorithm design to respect this constraint.

\begin{table}[tbp]
\centering
\caption{Hardware specifications for recent NVIDIA consumer GPU generations.}
\label{tab:gpu_specs}
\begin{tabular}{lccc}
\toprule
\textbf{Specification} & \textbf{RTX 3080} & \textbf{RTX 4090} & \textbf{RTX 5090} \\
\midrule
Architecture & Ampere & Ada Lovelace & Blackwell \\
Streaming Multiprocessors & 68 & 128 & 170 \\
CUDA Cores & 8,704 & 16,384 & 21,760 \\
Tensor Cores & 272 & 512 & 680 \\
Base Clock & 1.44 GHz & 2.23 GHz & 2.01 GHz \\
Boost Clock & 1.71 GHz & 2.52 GHz & 2.41 GHz \\
Memory Size & 10 GB GDDR6X & 24 GB GDDR6X & 32 GB GDDR7 \\
Memory Bandwidth & 760 GB/s & 1,008 GB/s & 1,792 GB/s \\
Shared Memory per SM & 128 KB & 128 KB & 128 KB \\
Max Threads per Block & 1,024 & 1,024 & 1,024 \\
\bottomrule
\end{tabular}
\end{table}

Table~\ref{tab:gpu_specs} summarizes the hardware specifications for the latest three generations of NVIDIA consumer GPUs. The progression from RTX 30 to RTX 50 series demonstrates significant architectural improvements across multiple dimensions. The SM count has increased substantially from 68 in the RTX 3080 to 170 in the RTX 5090, directly translating to higher parallel throughput capacity. This growth in SMs enables the execution of more concurrent thread blocks, making newer generations particularly effective for throughput-oriented workloads that can exploit massive parallelism. Concurrently, boost clock frequencies have also improved, rising from 1.71 GHz in the RTX 3080 to 2.41 GHz in the RTX 5090. Higher frequencies reduce the latency of individual operations, benefiting algorithms with sequential dependencies or limited parallelism. The combination of increased SM count and higher frequencies results in substantial improvements in both throughput and latency characteristics. Additionally, memory bandwidth has more than doubled from 760 GB/s to 1,792 GB/s, alleviating memory bottlenecks that often limit performance in data-intensive applications.

\subsection{Fused and Blocked Kernels} \label{sec:fused_and_blocked_kernels}

The warp library \cite{warp2022} provides a tile programming model that acts as a wrapper for the NVIDIA mathdx library, which includes cuBLASDx and cuSolverDx. These libraries are counterparts to the well-established cuBLAS and cuSolver libraries from NVIDIA, which provide optimized BLAS and LAPACK routines for NVIDIA GPUs at the host API level. In contrast, the mathdx libraries can be used as building blocks directly within CUDA kernels.

We implemented two different versions of the factorization and solve kernels: a fused variant and a blocked variant. The fused kernel implements Algorithm~\ref{alg:factorization_multi_stage} by fusing all operations in a single iteration (Lines 6 to 14) into one kernel. It loads data into shared memory, processes all operations, and then writes the results back to global memory, thereby minimizing memory transfers and kernel calls. However, this approach is limited by the available shared memory. For block sizes $n$ that are too large, the required shared memory exceeds what a streaming multiprocessor (SM) can provide.

To address this limitation, the blocked variant tiles the operations. For example, in the case of Cholesky factorization, instead of loading an entire matrix of size $n \times n$, we load only a submatrix of size $b \times b$ and process it tile by tile. Beyond matrix-level tiling, we further parallelize at the operation level as outlined in Section~\ref{sec:multi_stage_streams} and Algorithm~\ref{alg:factorization_multi_stage_right_looking_atomics}. Specifically, we execute operations in parallel across three CUDA streams. To further boost performance and reduce Python overhead, we use CUDA graphs to launch the kernel sequence. CUDA graphs enable recording and replaying CUDA kernel launches, thereby reducing the launch overhead of individual kernel invocations.

\begin{figure}[tbp]
\centering
\includegraphics[width=1.0\textwidth]{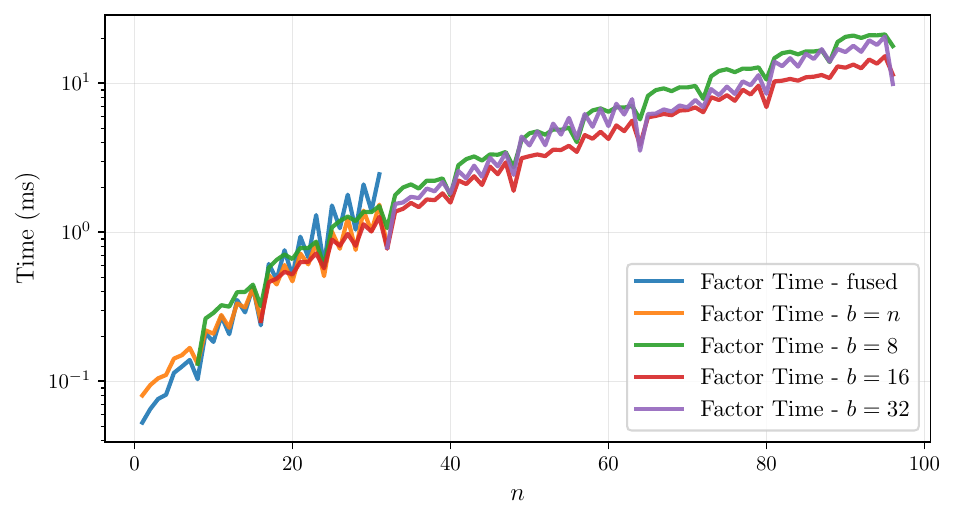}  
\caption{Computation time for different block sizes $n$ and horizon of $N=512$ on a RTX3080 using double precision (\texttt{f64}) for fused kernels and blocked kernels where $b$ is the tile size. For $n \geq 32$, the fused kernel runs out of shared memory.}
\label{fig:N512_rtx3080_block_size}
\end{figure}

Figure~\ref{fig:N512_rtx3080_block_size} shows the computation time for a horizon of $N=512$ across different block sizes $n$, comparing both the fused kernel variant and the blocked kernel variant with tile size $b$. The results use double precision (\texttt{f64}) with 64 threads per block. For small block sizes up to $n=16$, the fused kernels achieve the fastest factorization time due to lower overhead. Beyond $n=16$, the blocked variant becomes more efficient as operation-level parallelism provides speedup.

A notable quantization effect can be observed, i.e., odd block sizes are generally significantly slower. Figure~\ref{fig:N512_rtx3080_block_size_float32} shows the same experiment using single precision (\texttt{f32}), where this discretization effect manifests for block sizes divisible by four. This behavior highlights the importance of memory alignment. Matrices must be aligned to 128 bits to leverage NVIDIA GPUs' special instructions for loading 128-bit aligned memory.

Figures~\ref{fig:N512_rtx5090_block_size} and~\ref{fig:N512_rtx5090_block_size_float32} show analogous results for the RTX5090, exhibiting similar behavior. For double precision with $n$ between 16 and 32, setting $b=n$ (i.e., no tiling) is most efficient. However, as discussed earlier, tiling becomes necessary for larger block sizes to avoid exceeding shared memory limits. A tile size of $b=8$ appears too small to saturate the 64 available threads. While $b=16$ generally performs best and $b=32$ is superior when $n$ is a multiple of 32. Indeed, optimal performance is consistently achieved when $n$ is divisible by 8 for double precision and divisible by 16 for single precision. Therefore, padding the problem size accordingly ensures maximum performance

\begin{figure}[tbp]  
\centering  
\includegraphics[width=1.0\textwidth]{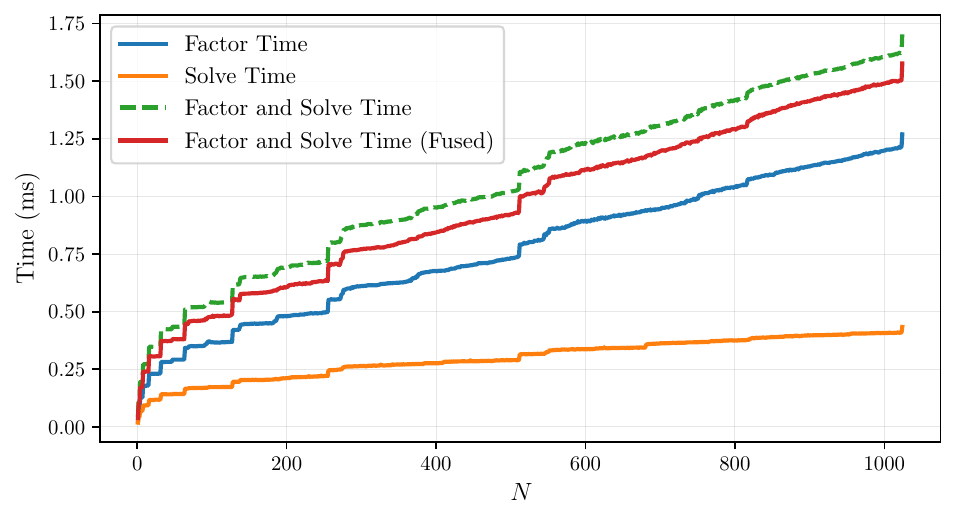}  
\caption{Computation time for different horizons $N$ with block size $n=32$ on a RTX3080 using double precision (\texttt{f64}).}
\label{fig:n32_rtx3080_times}  
\end{figure}

Figure~\ref{fig:n32_rtx3080_times} shows the computation times for different horizons $N$ with a block size of $n=32$ using the blocked kernel variant with $b=32$. Since it is common to factorize and then directly solve for a given right-hand side, we also implemented a fused variant where the factorization and forward substitution are interlaced, increasing both parallelism and performance. The interlaced kernel launches are illustrated in Figure~\ref{fig:fac_sol_timeline_both}. As shown in Figure~\ref{fig:n32_rtx3080_times}, this approach reduces the solve time significantly compared to running the factorization and solve routines sequentially.

The factorization time follows the theoretical $\mathcal{O}(\lfloor\log_2(N)\rfloor)$ complexity derived in Section~\ref{seq:mult_stage_perm} up to approximately $N=250$, but becomes more linear thereafter. A similar behavior is observed for the solve time, though this effect occurs later at around $N=700$. This departure from theoretical complexity arises because the SMs in the GPU become saturated, preventing the first iterations from being processed fully in parallel. This saturation occurs later for the solve operation because it requires less shared memory, allowing more blocks to be processed in parallel by a single SM.

\begin{figure}[tbp]  
\centering  
\includegraphics[width=1.0\textwidth]{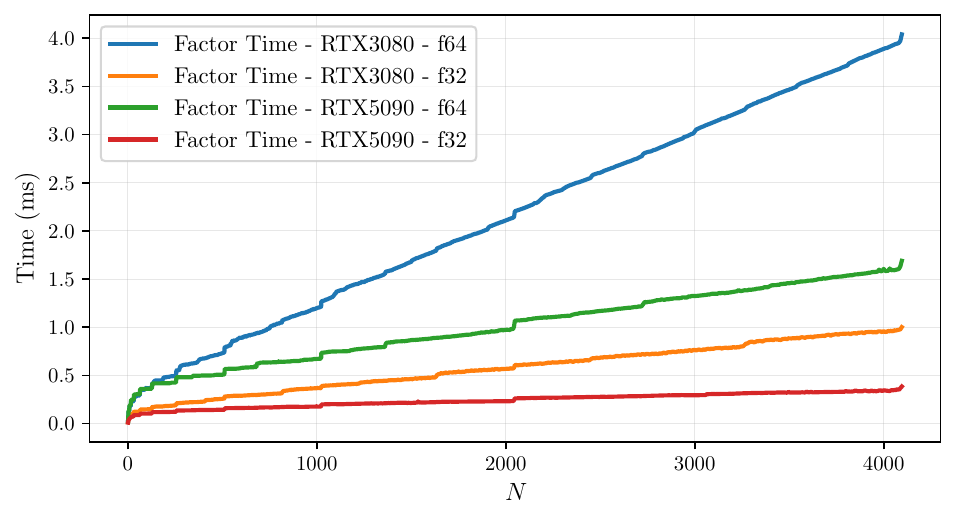}  
\caption{Computation time of the factorization for different horizons $N$ with block size $n=32$ on different GPUs and precisions.}
\label{fig:n32_rtx3080_rtx5090_factor_times}  
\end{figure}

Figure~\ref{fig:n32_rtx3080_rtx5090_factor_times} compares factorization times on the RTX3080 and RTX5090 using single (\texttt{f32}) and double (\texttt{f64}) precision. For small horizons, the factorization times are very similar between the RTX3080 and RTX5090. However, the RTX5090 considerably outperforms the RTX3080 for larger horizons $N$, achieving roughly a $2.5\times$ speedup, which aligns with its $2.5\times$ greater number of SMs as outlined in Table~\ref{tab:gpu_specs}. We also observe that single precision exhibits a $4\times$ speedup compared to double precision. This can be explained by both the lower shared memory requirements and the fact that single precision floating point operations are 1.5--2$\times$ faster than double precision floating point operations \cite{jia2019dissectingnvidiaturingt4}. A similar behavior can be seen for the solve times in Figure~\ref{fig:n32_rtx3080_rtx5090_solve_times}.

\section{Comparison to other CPU and GPU Solvers} \label{sec:numerical_examples_gpu}

\subsection{QDLDL}

\begin{figure}[tbp]  
\centering  
\includegraphics[width=1.0\textwidth]{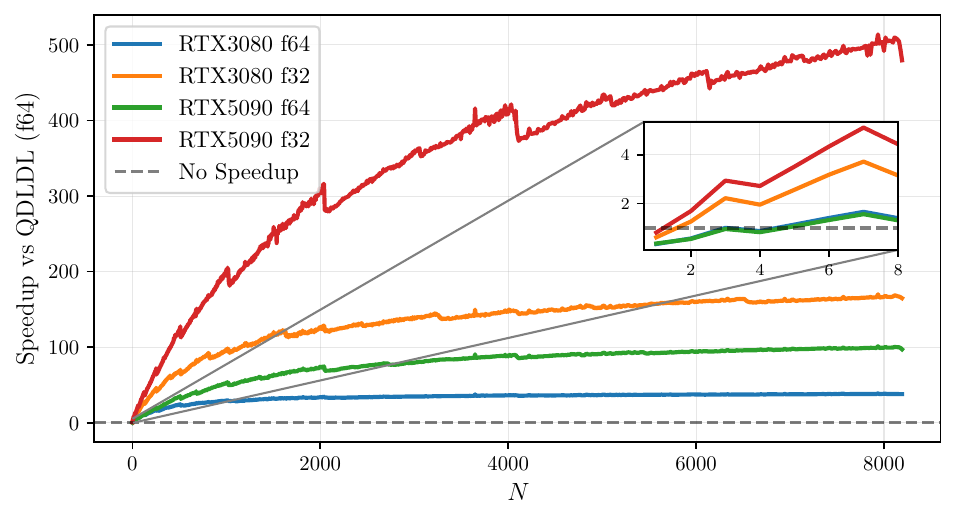}  
\caption{Speedup of the factorization time for different horizons $N$ with block size $n=32$ in respect to QDLDL with different GPUs and precisions. QDLDL only provides interfaces with double floating-point precision.}
\label{fig:n32_fac_speedup_qdldl}  
\end{figure}

In a first step, we compare against the sparse LDL solver QDLDL, which is part of the popular QP solver OSQP \cite{stellato2020}. The reported timings for QDLDL include only the numerical factorization time, excluding the symbolic factorization phase.

Figure~\ref{fig:n32_fac_speedup_qdldl} shows the speedup of the factorization routine for increasing horizon $N$ across different GPUs and precisions. Considerable speedups are observed as the horizon increases, with speedups evident even for small horizons. As the GPU saturates and computation time transitions from the $\mathcal{O}(\lfloor\log_2(N)\rfloor)$ regime to a linear one, the maximum speedup plateaus at approximately $40\times$ for the RTX3080 and $100\times$ for the RTX5090 using double precision. With single precision, maximum speedups of $170\times$ for the RTX3080 and over $500\times$ for the RTX5090 are achieved.

Figure~\ref{fig:n32_solve_speedup_qdldl} shows similar behavior for the solve times, though the speedups are less pronounced and saturation occurs at longer horizons.

\subsection{BLASFEO}

\begin{figure}[tbp]  
\centering  
\includegraphics[width=1.0\textwidth]{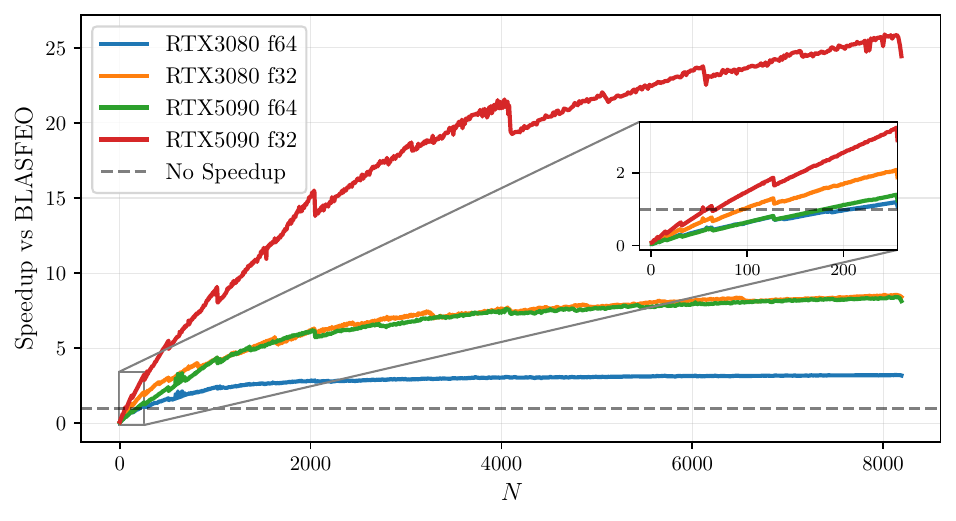}  
\caption{Speedup of the factorization time for different horizons $N$ with block size $n=32$ in respect to BLASFEO with different GPUs and precisions.}
\label{fig:n32_fac_speedup_blasfeo}  
\end{figure}

To ensure a fair comparison, we also implemented Algorithm~\ref{alg:seq_factorization} on the CPU using the BLASFEO library \cite{frison2018} as a backend, rather than comparing solely against the more general-purpose sparse LDL solver QDLDL. We created Python bindings for the experiments, passing data by reference to avoid introducing overhead.

Figure~\ref{fig:n32_fac_speedup_blasfeo} shows the speedups achieved by our GPU solver. Note that here, we compare implementations with the same floating-point precision, e.g., the \texttt{f64} GPU implementation with the \texttt{f64} BLASFEO implementation. The BLASFEO backend significantly outperforms QDLDL due to its use of AVX2 (Single Instruction, Multiple Data) SIMD instructions, which enable the CPU to process up to four doubles or eight floats simultaneously, providing limited instruction-level parallelism. Additionally, BLASFEO exhibits better data locality, leading to improved cache utilization. Since memory transfers constitute a substantial bottleneck on modern CPUs, this enhanced locality provides further speedup.

Nevertheless, our GPU implementation achieves over $25\times$ speedup with an RTX5090 using single precision. Even with double precision on an RTX3080, we still observe a $2\times$ speedup for long horizons $N$. The crossover point occurs at horizons around $N=100$ to $N=200$.

\subsection{CUDSS}

\begin{figure}[tbp]
\centering
\includegraphics[width=1.0\textwidth]{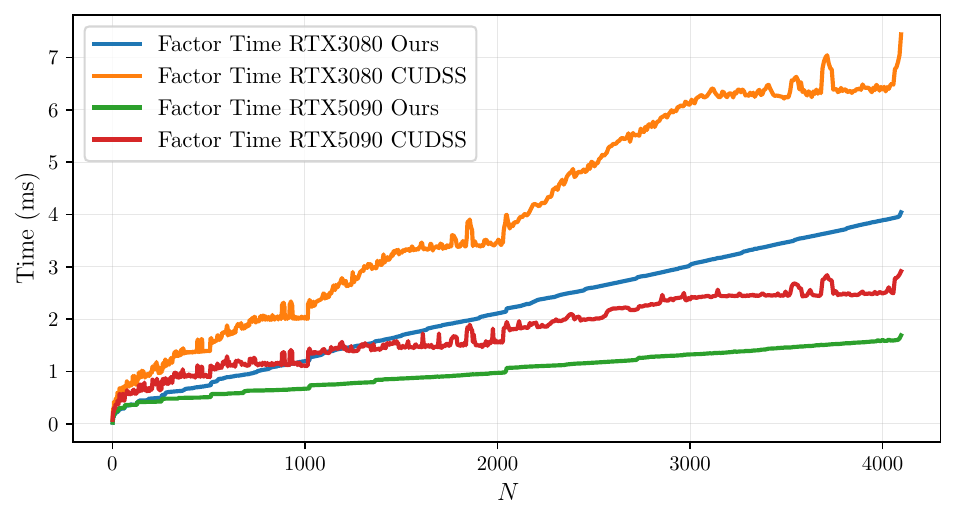}
\caption{Comparison between factorization times of our factorization method and CUDSS for different horizons $N$ with block size $n=32$ on different GPUs using double precision.}
\label{fig:n32_cudss_fact}
\end{figure}

We also compare our method to NVIDIA's closed-source CUDSS library. Figure~\ref{fig:n32_cudss_fact} and Figure~\ref{fig:n32_cudss_solve} show the factorization and solve times, respectively, on an RTX3080 and an RTX5090 using double precision. Our method consistently achieves approximately 2$\times$ speedup across both metrics.

While the exact algorithm deployed in CUDSS is not documented, we believe the permutation in their analysis phase likely employs a nested dissection method, resulting in similar asymptotic scaling to our approach. Indeed, the jumps in computation time at powers-of-two boundaries suggest that CUDSS also exploits the tree structure induced by nested dissection permutations. The performance difference primarily stems from CUDSS's inability to exploit the otherwise dense structure inherent to our problem.

Additionally, the plots do not include symbolic analysis time, which is typically two orders of magnitude slower than factorization time. This phase encompasses permutation calculation and symbolic factorization to prepare for numerical factorization. We exclude this overhead from our comparison because the matrix structure is known a priori, allowing the symbolic analysis to be performed offline.

\section{Conclusion} \label{sec:conclusion}

We have presented a comprehensive framework for GPU-accelerated Cholesky factorization of block tridiagonal matrices through systematic matrix reordering strategies. Our multi-stage permutation approach, similar to nested dissection, reduces computational complexity from $\mathcal{O}(Nn^3)$ to $\mathcal{O}(\log_2(N)n^3)$ when sufficient parallel resources are available. By exploiting both column-level and operation-level parallelism through CUDA streams and atomic operations, we further reduced the critical path length and improved practical performance.

The numerical implementation demonstrated substantial speedups on modern NVIDIA GPUs. Compared to the sparse LDL solver QDLDL, our approach achieved speedups exceeding $100\times$ on an RTX 3080 and over $500\times$ on an RTX 5090 using single precision for long horizons. Even against the highly optimized CPU implementation tailored for block tridiagonal systems using BLASFEO with SIMD instructions, our GPU solver maintained speedups of $25\times$ to $40\times$ for single precision, with crossover points occurring around $N=100$ to $N=200$. Notably, our method also outperforms NVIDIA's closed-source CUDSS library by approximately $2\times$ for both factorization and solve operations. This performance advantage stems from our algorithm's ability to exploit the dense block structure inherent to block tridiagonal matrices, which general-purpose sparse solvers cannot leverage as effectively.

Our approach is particularly well-suited for applications in model predictive control, Kalman filtering, and other domains where block tridiagonal systems with moderately long horizons arise frequently. This also includes the integration into optimization solver for applications in robotics and power systems. The logarithmic scaling ensures that performance continues to improve as horizon length increases, making it an attractive solution for real-time applications requiring repeated solution of large structured linear systems. Future work could explore more aggressive kernel fusion, extension to block banded matrices with larger bandwidth, mixed-precision strategies to further enhance performance.

\ifisThesis\else
\bibliographystyle{IEEEtran}
\bibliography{tail/refs}
\fi

\newpage
\begin{subappendices}

\begin{figure}[tbp]  
\centering  
\includegraphics[width=1.0\textwidth]{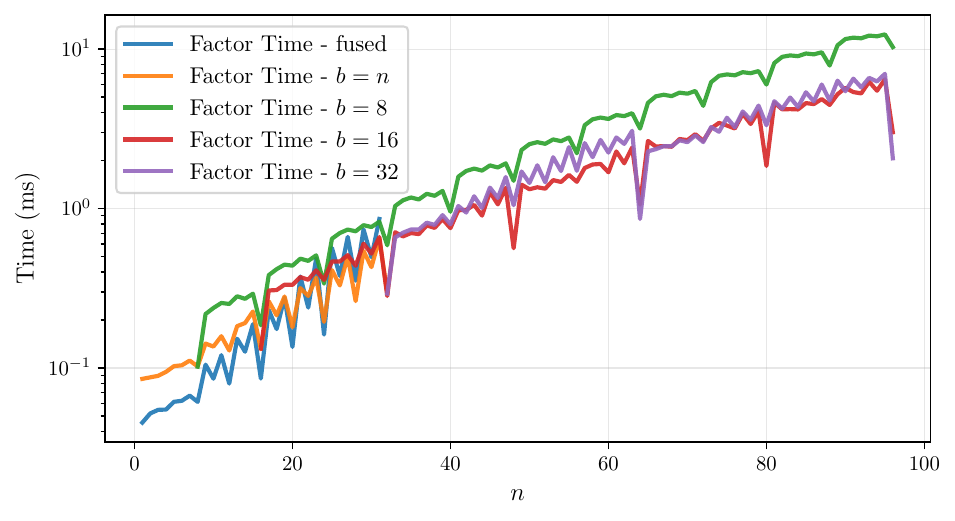}  
\caption{Computation time for different block sizes $n$ and horizon of $N=512$ on a RTX3080 using double precision (\texttt{f32}) for fused kernels and blocked kernels where $b$ is the tile size.}
\label{fig:N512_rtx3080_block_size_float32}  
\end{figure}

\begin{figure}[tbp]  
\centering  
\includegraphics[width=1.0\textwidth]{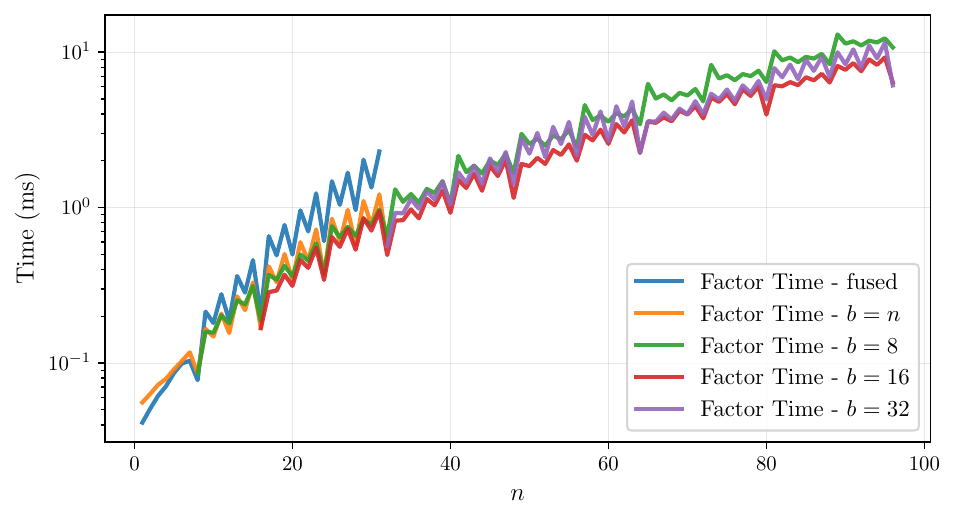}  
\caption{Computation time for different block sizes $n$ and horizon of $N=512$ on a RTX5090 using double precision (\texttt{f64}) for fused kernels and blocked kernels where $b$ is the tile size.}
\label{fig:N512_rtx5090_block_size}  
\end{figure}

\begin{figure}[tbp]  
\centering  
\includegraphics[width=1.0\textwidth]{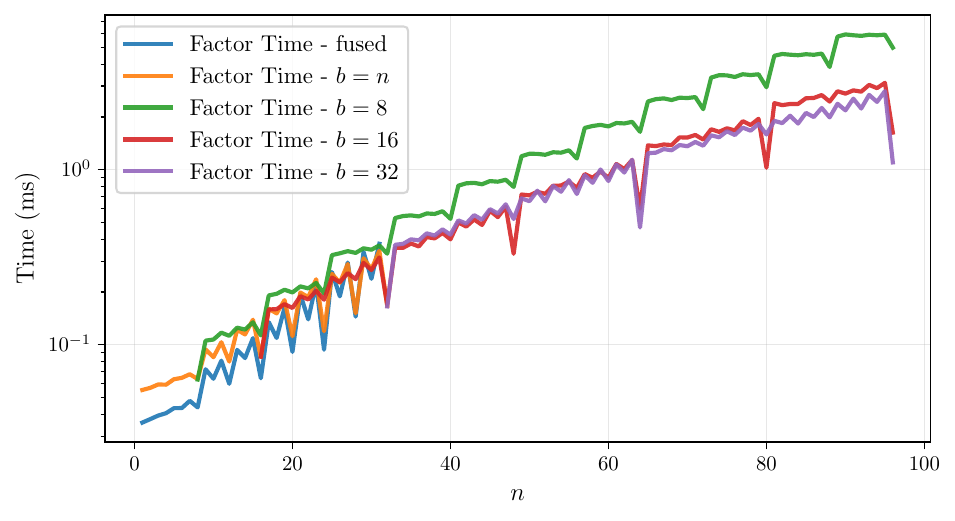}  
\caption{Computation time for different block sizes $n$ and horizon of $N=512$ on a RTX5090 using double precision (\texttt{f32}) for fused kernels and blocked kernels where $b$ is the tile size.}
\label{fig:N512_rtx5090_block_size_float32}  
\end{figure}

\begin{figure}[tbp]  
\centering  
\includegraphics[width=1.0\textwidth]{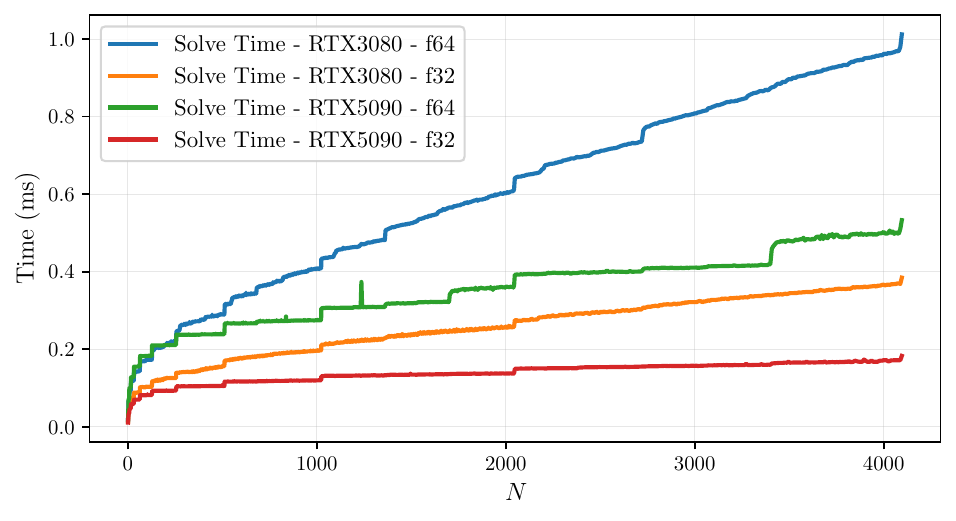}  
\caption{Computation time of the solve for different horizons $N$ with block size $n=32$ on different GPUs and precisions.}
\label{fig:n32_rtx3080_rtx5090_solve_times}  
\end{figure}

\begin{figure}[tbp]  
\centering  
\includegraphics[width=1.0\textwidth]{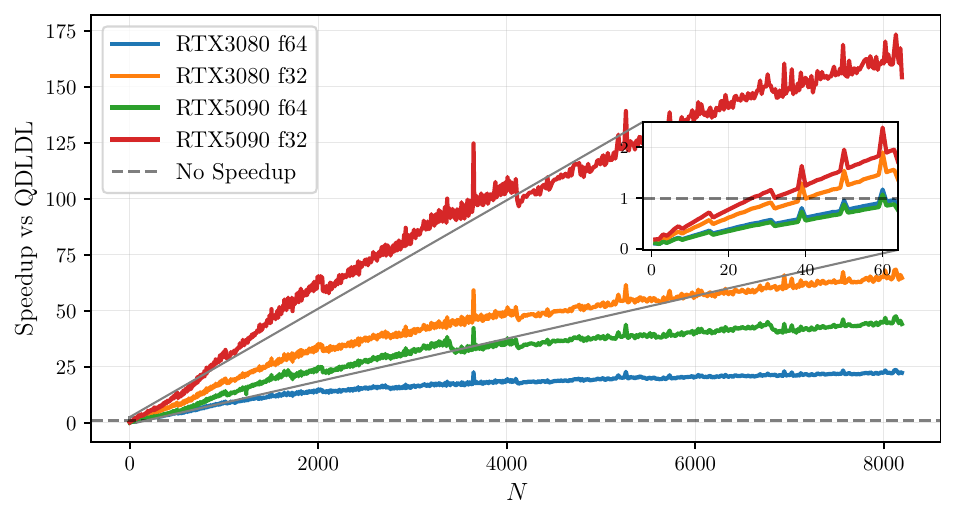}  
\caption{Speedup of the solve time for different horizons $N$ with block size $n=32$ in respect to QDLDL with different GPUs and precisions.}
\label{fig:n32_solve_speedup_qdldl}  
\end{figure}

\begin{figure}[tbp]
\centering
\includegraphics[width=1.0\textwidth]{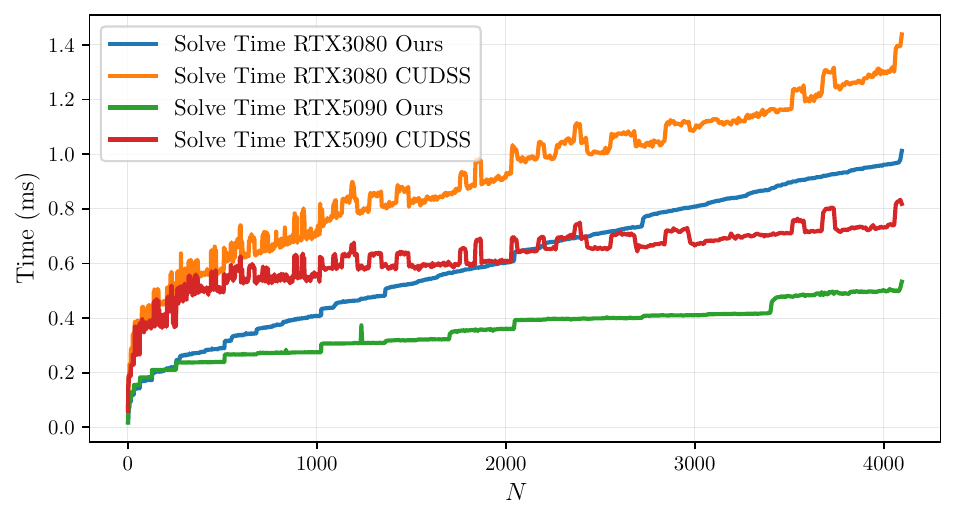}
\caption{Comparison between solve times of our factorization method and CUDSS on different GPUs using double precision.}
\label{fig:n32_cudss_solve}
\end{figure}

\end{subappendices}

\end{document}